\documentclass[10pt]{article}
\usepackage{amsmath, amsthm, amssymb, amsfonts, mathrsfs, mathtools}
\usepackage{graphicx}
\usepackage{makeidx}
\usepackage[left=2cm,right=2cm,top=2cm,bottom=2cm]{geometry}
\usepackage{caption}
\usepackage{subcaption}
\usepackage[section]{placeins}
\usepackage{enumitem}
\usepackage{tikz}
\usepackage{stmaryrd}
\usepackage{makecell}
\usepackage{hhline}

\usetikzlibrary{decorations.pathreplacing}
\tikzstyle{vertex}=[auto=left,circle,draw=black,fill=white, inner sep=1.5]

\usepackage{hyperref}
\hypersetup{colorlinks=true, citecolor={blue}, linkcolor=blue, filecolor=magenta, urlcolor=cyan}

\newtheorem{theorem}{Theorem}[section]

\newtheorem{lema}[theorem]{Lemma}

\newtheorem{ex}{Example}[section]

\title{Laplacian Pair State Transfer on Total Graphs}

\author{ Akash Kalita and Bikash Bhattacharjya\\
Department of Mathematics\\
Indian Institute of Technology Guwahati, India\\
akash.kalita@iitg.ac.in, b.bikash@iitg.ac.in }

\date{}
\begin{document}
\maketitle

\vspace{-0.3in}

\begin{center}{\textbf{Abstract}}\end{center}
The total graph of a graph $G$, denoted $\mathcal{T}(G)$, is defined as the graph whose vertex set is the union of the vertex set of $G$ and the edge set of $G$, such that two vertices of $\mathcal{T}(G)$ are adjacent if the corresponding elements of $G$ are either adjacent or incident. In this paper, we investigate the existence of Laplacian perfect pair state transfer and Laplacian pretty good pair state transfer on $\mathcal{T}(G)$, where $G$ is an $r$-regular graph. We prove that if $G$ is Laplacian integral, $r \geq 3$, and $r+1$ is not a Laplacian eigenvalue of $G$, then $\mathcal{T}(G)$ does not exhibit Laplacian perfect pair state transfer. In addition, we prove that under some mild conditions, $\mathcal{T}(G)$ exhibits Laplacian pretty good pair state transfer, where $r \geq 3$ and $r+1$ is not a Laplacian eigenvalue of $G$. Using these conditions, we obtain several infinite families of total graphs exhibiting Laplacian pretty good pair state transfer that fail to exhibit Laplacian perfect pair state transfer. We also prove that the total graph of the complete graph $K_n$ exhibits Pair-LPGST if and only if $n=3$.

\noindent 
\textbf{Keywords:} Total graph, Laplacian perfect state transfer, Laplacian pretty good state transfer, Laplacian perfect pair state transfer, Laplacian pretty good pair state transfer.  
\\
\textbf{Mathematics Subject Classifications:} 15A16, 05C50, 81P68  
\section{Introduction}
Let $G$ be a finite simple connected undirected graph. Let $A_G$ and $L_G$ be the adjacency matrix and the  Laplacian matrix of $G$, respectively. Also, let $t$ be a real number and $\mathbf{i} := \sqrt{-1}$. The \textit{transition matrix} of $G$ with respect to $A_G$, denoted $H_G(t)$, is defined as 
\begin{align*}
H_G(t) :=  \exp(-\mathbf{i}t A_G) = \sum_{j = 0}^{\infty} \frac{(-\mathbf{i}t A_G)^j}{j!}. 
\end{align*}
The transition matrix of $G$ with respect to $L_G$, denoted $U_G(t)$, is defined as 
\begin{align*}
U_G(t) :=  \exp(-\mathbf{i}t L_G) = \sum_{j = 0}^{\infty} \frac{(-\mathbf{i}t L_G)^j}{j!}. 
\end{align*}
Let $n$ be the number of vertices of the graph $G$ and $a$ be a vertex of $G$. Let $\textbf{e}^n_a$ denote the column matrix indexed by the vertices of $G$ such that the $a$-th entry of $\textbf{e}^n_a$ is 1, and 0 elsewhere. A graph $G$ exhibits \textit{perfect state transfer} (in short PST) between distinct vertices $a$ and $b$ if there exists a non-zero real number $t$ such that 
$$\left|{\textbf{e}^n_{a}}^TH_G(t){\textbf{e}^n_{b}}\right| = 1.$$ 
A graph $G$ exhibits \textit{pretty good state transfer} (in short PGST) between distinct vertices $a$ and $b$ if for every positive real number $\varepsilon$, there exists a non-zero real number $t$ such that 
$$\left|\left|{{\textbf{e}^n_{a}}}^TH_G(t){\textbf{e}^n_{b}} \right| - 1\right|< \varepsilon.$$
In the definition of PST (PGST), if we replace $H_G(t)$ by $U_G(t)$, then PST (PGST) is called LPST (LPGST). If $a$ and $b$ are two distinct vertices of $G$, then the \textit{pair state} associated with $(a, b)$ is defined by the column matrix 
\begin{align*}
\frac{1}{\sqrt{2}}(\textbf{e}^n_{a} - \textbf{e}^n_{b}).
\end{align*}
For convenience, we exclude the scalar $\frac{1}{\sqrt{2}}$ and use $\textbf{e}^n_{a} - \textbf{e}^n_{b}$ to represent a pair state. A graph $G$ exhibits \textit{Laplacian perfect pair state transfer} (in short Pair-LPST) between two linearly independent  pair states ${\textbf{e}^n_{a}} - {\textbf{e}^n_{b}}$ and ${\textbf{e}^n_{c}} - {\textbf{e}^n_{d}}$ if there exists a non-zero real number $t$ such that 
$$\left|\frac{1}{2}{({\textbf{e}^n_{a}} - {\textbf{e}^n_{b}})}^TU_G(t)({\textbf{e}^n_{c}} - {\textbf{e}^n_{d}})\right| = 1.$$ 
In particular, if $a,b$ are adjacent and $c,d$ are adjacent in $G$, then Laplacian perfect pair state transfer is termed as \textit{Laplacian perfect edge state transfer}  (in short Edge-LPST). 
In $2020$, Chen and Godsil~\cite{QChen2020} introduced the notion of Pair-LPST, which is an analogue of LPST. However, Pair-LPST does not always exist on graphs. In $2023$, Wang et al.~\cite{Vertexcoronapairstatetransfer2023} introduced the notion of Laplacian pretty good pair state transfer, a relaxation of Laplacian perfect pair state transfer. A graph $G$ exhibits \textit{Laplacian pretty good pair state transfer} (in short Pair-LPGST) between two linearly independent  pair states ${\textbf{e}^n_{a}} - {\textbf{e}^n_{b}}$ and ${\textbf{e}^n_{c}} - {\textbf{e}^n_{d}}$ if for every positive real number $\varepsilon$, there exists a non-zero real number $t$ such that 
$$\left|\left|\frac{1}{2}{({\textbf{e}^n_{a}} - {\textbf{e}^n_{b}})}^TU_G(t)({\textbf{e}^n_{c}} - {\textbf{e}^n_{d}})\right| - 1\right|< \varepsilon.$$ 
In particular, if $a,b$ are adjacent and $c,d$ are adjacent in $G$, then Laplacian pretty good pair state transfer is termed as \textit{Laplacian pretty good edge state transfer} (in short Edge-LPGST). 

Chen and Godsil~\cite{QChen2020} provided a theoretical introduction to Pair-LPST and provided methods to construct infinite families of graphs exhibiting Pair-LPST. They also classified the paths and cycles exhibiting Pair-LPST. Cao~\cite{Xiwang2021}, Cao and Wan~\cite{XiwangCao2022} and Luo et al.~\cite{GaojunLuo2022} studied Edge-LPST on cubelike graphs, abelian Cayley graphs and normal Cayley graphs over dihedral groups, respectively. 
In ~\cite{Vertexcoronapairstatetransfer2023}, Wang et al. studied Pair-LPST and Pair-LPGST on vertex coronas. In~\cite{PairStateTransferQGraph2025}, Jiang et al. studied Pair-LPST and Pair-LPGST on $\mathcal{Q}$-graphs of regular graphs. Recently, Jiang et al.~\cite{DoubleCover2026} explored Pair-LPST on tensor product and double cover of graphs. Wang et al.~\cite{Laplacianprettygoodedgestatetransferinpaths2022} provided necessary and sufficient conditions for the existence of Edge-LPGST on graphs. Using those conditions, they classified the paths exhibiting Edge-LPGST.  

Cvetkovic et al.~\cite{AnIntroductiontotheTheoryofGraphSpectra2010} introduced the notion of the total graph of a graph. The \textit{total graph} of a graph $G$, denoted $\mathcal{T}(G)$, is defined as the graph whose vertex set is the union of the vertex set of $G$ and the edge set of $G$, such that two vertices of $\mathcal{T}(G)$ are adjacent if the corresponding elements of $G$ are adjacent or incident. Liu and Wang~\cite{TotalGraph2021} studied LPST and LPGST on $\mathcal{T}(G)$, where $G$ is an $r$-regular graph. They proved that $\mathcal{T}(G)$ does not exhibit LPST if $r+1$ is not a Laplacian eigenvalue of $G$. If $r+1$ is not a Laplacian eigenvalue of $G$, then Liu and Wang~\cite{TotalGraph2021} provided some sufficient conditions for the existence of LPGST on $\mathcal{T}(G)$. Using those sufficient conditions, they obtain an infinite family of $\mathcal{T}(G)$ exhibiting LPGST that fail to exhibit LPST. Further, they asked the question: if $G$ is $r$-regular and $r+1$ is a Laplacian eigenvalue of $G$, can $\mathcal{T}(G)$ exhibit LPST? In $2023$, Liu and Liu~\cite{TotalGraph2023} answered the preceding question and proved that $\mathcal{T}(G)$ does not exhibit LPST for any regular graph $G$. Inspired by the preceding work, we consider the existence of Pair-LPST and Pair-LPGST on $\mathcal{T}(G)$, where $G$ is an $r$-regular graph.

The subsequent sections of this paper are structured as follows. In Section~\ref{preliminary}, we introduce some basic definitions and preliminary results that we use in the rest of the sections. In Section~\ref{perfect}, we explore the existence of Pair-LPST on $\mathcal{T}(G)$, where $G$ is an $r$-regular graph such that $r \geq 3$. We prove that $\mathcal{T}(G)$ does not exhibit Pair-LPST if $G$ is Laplacian integral and $r+1$ is not a Laplacian eigenvalue of $G$. In Section~\ref{pretty good}, we explore the existence of Pair-LPGST on $\mathcal{T}(G)$, where $G$ is an $r$-regular graph such that $r \geq 3$. We prove that $\mathcal{T}(G)$ exhibits Pair-LPGST under some mild conditions. In Section~\ref{Concrete examples}, we obtain several infinite families of $\mathcal{T}(G)$ exhibiting Pair-LPGST using the conditions obtain in Section~\ref{pretty good}. Moreover, we use Cartesian product and tensor product of graphs to obtain a few more infinite families of $\mathcal{T}(G)$ exhibiting Pair-LPGST. Finally, we prove that $\mathcal{T}(K_n)$ exhibits Pair-LPGST if and only if $n=3$. 

\section{Basic Definitions and Preliminary Results}\label{preliminary}
In this section, we discuss some notations, basic definitions and results which we use in the rest of the paper. All the graphs considered in this paper are finite, simple, undirected and connected. Let $m$ and $n$ be two positive integers. Throughout this paper, $J_{m \times n}$ denotes the $m \times n$ all one matrix, $\textbf{0}_{m \times n}$ denotes the $m \times n$ zero matrix, $I_{n \times n}$ denotes the $n \times n$ identity matrix, $\textbf{0}_{n}$ denotes the zero column matrix with $n$ entries, and $\textbf{1}_{n}$ denotes the all one column matrix with $n$ entries. Also, $K_n$ denotes the complete graph, $Q_n$ denotes the $n$-cube and $\mathrm{Cay}(\Gamma, S)$ denotes the Cayley graph over the group $\Gamma$ with connection set $S$. The \textit{dihedral group}, denoted $D_{2n}$, is defined as $D_{2n} := \langle u,v \colon u^{n} = v^2 = \textbf{1}, vuv = u^{-1}\rangle$.  The \textit{dicyclic group}, denoted $T_{4n}$, is defined as $T_{4n} := \langle u,v \colon u^{2n} = \textbf{1}, u^n = v^2, v^{-1}uv = u^{-1}\rangle$. The \textit{semi-dihedral group}, denoted $SD_{8n}$, is defined as $SD_{8n} := \langle u,v \colon u^{4n} = v^2 = \textbf{1}, vuv = u^{2n-1}\rangle$.

Let $G$ be a graph on $n$ vertices. Also, let $L_G$ be the Laplacian matrix of $G$ and $\theta_0,\theta_1, \ldots, \theta_p$ be the Laplacian eigenvalues such that $0=\theta_0<\theta_1 < \cdots < \theta_p$. If $E_{\theta_0}{(G)},E_{\theta_1}{(G)}, \ldots, E_{\theta_p}{(G)}$ are the Laplacian eigenprojectors corresponding to the eigenvalues $\theta_0,\theta_1, \ldots, \theta_p$, then the spectral decomposition of $L_G$ is 
$$L_G = \sum_{j=0}^{p}\theta_jE_{\theta_j}{(G)}.$$
Let $\textbf{e}^n_a - \textbf{e}^n_b$ and $\textbf{e}^n_c - \textbf{e}^n_d$ be two linearly independent pair states of $G$. The \textit{Laplacian eigenvalue support} of $\textbf{e}^n_a - \textbf{e}^n_b$, denoted $\Phi_{ab}$, is defined as the set $\{\theta_j \colon E_{\theta_j}{(G)}(\textbf{e}^n_a - \textbf{e}^n_b) \neq \textbf{0}_n \}$. The pair states $\textbf{e}^n_a - \textbf{e}^n_b$ and $\textbf{e}^n_c - \textbf{e}^n_d$ are called \textit{Laplacian strongly cospectral} if $E_{\theta_j}{(G)}(\textbf{e}^n_a - \textbf{e}^n_b) = \pm E_{\theta_j}{(G)}(\textbf{e}^n_c - \textbf{e}^n_d)~~\mathrm{for}~0 \leq j \leq p$. If $\textbf{e}^n_a - \textbf{e}^n_b$ and $\textbf{e}^n_c - \textbf{e}^n_d$ are Laplacian strongly cospectral, then we consider the sets $\Phi^{+}_{ab,cd}$ and $\Phi^{-}_{ab,cd}$ such that 
\begin{align*}
\Phi^{+}_{ab,cd} &= \{\theta_j \colon E_{\theta_j}{(G)}(\textbf{e}^n_a - \textbf{e}^n_b) = E_{\theta_j}{(G)}(\textbf{e}^n_c - \textbf{e}^n_d) \neq \textbf{0}_n\}~\mathrm{and}~\\
\Phi^{-}_{ab,cd} &= \{\theta_j \colon E_{\theta_j}{(G)}(\textbf{e}^n_a - \textbf{e}^n_b) = -E_{\theta_j}{(G)}(\textbf{e}^n_c - \textbf{e}^n_d) \neq \textbf{0}_n\}.
\end{align*}
It follows that for Laplacian strongly cospectral pair states $\textbf{e}^n_a - \textbf{e}^n_b$ and $\textbf{e}^n_c - \textbf{e}^n_d$, we have 
$$\Phi_{ab} = \Phi_{cd} = \Phi^{+}_{ab,cd} \cup \Phi^{-}_{ab,cd}.$$ 
It is clear that the set $\Phi_{ab}$ is non-empty for any pair state $\textbf{e}^n_a - \textbf{e}^n_b$ of $G$. The state $\textbf{e}^n_a - \textbf{e}^n_b$ is called a \textit{fixed state} if $\Phi_{ab}$ is a singleton set. Jiang et al.~\cite{
PairStateTransferQGraph2025} proved the following lemma.
\begin{lema}\emph{\cite{PairStateTransferQGraph2025}}\label{fixed state}
Let $\textbf{e}^n_a - \textbf{e}^n_b$ and $\textbf{e}^n_c - \textbf{e}^n_d$ be two Laplacian strongly cospectral pair states of a graph $G$. Then both $\Phi^{+}_{ab,cd}$ and $\Phi^{-}_{ab,cd}$ are non-empty, that is, $\textbf{e}^n_a - \textbf{e}^n_b$ and $\textbf{e}^n_c - \textbf{e}^n_d$ cannot be fixed.
\end{lema}
Note that Lemma~\ref{fixed state} also follows from Godsil et al.~\cite{RealStateTransfer2025}. Chen and Godsil~\cite{QChen2020} obtained necessary and sufficient conditions for the existence of Pair-LPST.
\begin{theorem}\emph{\cite{QChen2020}}\label{Conditions for the existence of perfect pair state transfer}
Let $\textbf{e}^n_a - \textbf{e}^n_b$ and $\textbf{e}^n_c - \textbf{e}^n_d$ be two linearly independent pair states of a graph $G$. Further, let $\Phi_{ab}$ be the set $\{\theta_0, \ldots, \theta_{\ell}\}$ such that $\theta_0 \in \Phi^{+}_{ab,cd}$. Then $G$ exhibits Pair-LPST between $\textbf{e}^n_a - \textbf{e}^n_b$ and $\textbf{e}^n_c - \textbf{e}^n_d$  if and only if the following three conditions hold.
\begin{enumerate}
\item[(i)] The pair states $\textbf{e}^n_a - \textbf{e}^n_b$ and $\textbf{e}^n_c - \textbf{e}^n_d$ are Laplacian strongly cospectral.
\item[(ii)] Either $\Phi_{ab} \subseteq \mathbb{Z}$ or there is an integer $x$ and a square-free integer $\Delta$ strictly greater than one such that 
$$\Phi_{ab} \subseteq \left\{\frac{1}{2}(x + y \sqrt{\Delta}) \colon ~y~\mathrm{is~an~even~integer}\right \} ~~\mathrm{or}~~ \Phi_{ab} \subseteq \left\{\frac{1}{2}(x + y \sqrt{\Delta}) \colon ~y~\mathrm{is~an~odd~integer}\right \}.$$
\item[(iii)] Let \[g= \left\{ \begin{array}{rl} \gcd \left(\{\theta_0-\theta_j\}_{j=0}^{\ell}\right) & \textrm{ if } ~\Phi_{ab} \subseteq \mathbb{Z}\\
 \gcd \left(\left\{ \frac{\theta_0-\theta_j}{\sqrt{\Delta}}\right\}_{j=0}^{\ell}\right) & \textrm{ if }~\Phi_{ab} \subseteq \left\{\frac{1}{2}(x + y \sqrt{\Delta}) \colon ~y~\mathrm{is~an~integer}\right \},\\
\end{array}\right. \]
where $x$ is an integer and $\Delta$ is a square-free integer strictly greater than one.
Then

$(a)~\theta_j \in \Phi^{+}_{ab,cd}$ if and only if 
\[ \left\{ \begin{array}{rl} \frac{\theta_0-\theta_j}{g}  \text{ is even}  & \textrm{ if } ~\Phi_{ab} \subseteq \mathbb{Z}\\
\frac{\theta_0-\theta_j}{g\sqrt{\Delta}} \text{ is even}  & \textrm{ if }~\Phi_{ab} \subseteq \left\{\frac{1}{2}(x + y \sqrt{\Delta}) \colon ~y~\mathrm{is~an~integer}\right \},\\
\end{array}\right. \]

$(b)~\theta_j \in \Phi^{-}_{ab,cd}$ if and only if 
\[ \left\{ \begin{array}{rl} \frac{\theta_0-\theta_j}{g}  \text{ is odd}  & \textrm{ if } ~\Phi_{ab} \subseteq \mathbb{Z}\\
\frac{\theta_0-\theta_j}{g\sqrt{\Delta}} \text{ is odd}  & \textrm{ if }~\Phi_{ab} \subseteq \left\{\frac{1}{2}(x + y \sqrt{\Delta}) \colon ~y~\mathrm{is~an~integer}\right \},\\
\end{array}\right. \]

\end{enumerate}
If the conditions $(i)$, $(ii)$ and $(iii)$ hold, then $G$ exhibits Pair-LPST at minimum $t_0$, where 
\[t_0 = \left\{ \begin{array}{cl} \frac{\pi}{g }  & \textrm{ if } ~\Phi_{ab} \subseteq \mathbb{Z}\\
 \frac{\pi}{g \sqrt{\Delta}} & \textrm{ if }~\Phi_{ab} \subseteq \left\{\frac{1}{2}(x + y \sqrt{\Delta}) \colon ~y~\mathrm{is~an~integer}\right \},\\
\end{array}\right. \]
\end{theorem}
The next theorem provides another necessary and sufficient conditions for the existence of Pair-LPST.
\begin{theorem}\emph{\cite{QChen2020}}\label{More conditions for the existence of perfect pair state transfer}
Let $\textbf{e}^n_a - \textbf{e}^n_b$ and $\textbf{e}^n_c - \textbf{e}^n_d$ be two linearly independent pair states of a graph $G$, and $\theta_0 \in \Phi^{+}_{ab,cd}$. Then $G$ exhibits Pair-LPST between $\textbf{e}^n_a - \textbf{e}^n_b$ and $\textbf{e}^n_c - \textbf{e}^n_d$ at $t_1$ if and only if the following three conditions hold.
\begin{enumerate}
\item[(i)] The pair states $\textbf{e}^n_a - \textbf{e}^n_b$ and $\textbf{e}^n_c - \textbf{e}^n_d$ are Laplacian strongly cospectral.
\item[(ii)] There exists an integer $k_1$ such that $t_1(\theta_0 - \theta_j) = 2k_1 \pi$ for all $\theta_j \in \Phi^{+}_{ab,cd}$. 
\item[(iii)] There exists an integer $k_2$ such that $t_1(\theta_0 - \theta_j) = (2k_2+1) \pi$ for all $\theta_j \in \Phi^{-}_{ab,cd}$. 
\end{enumerate}
\end{theorem}
Let $G$ exhibit Pair-LPST between $\textbf{e}^n_a - \textbf{e}^n_b$ and $\textbf{e}^n_c - \textbf{e}^n_d$ at minimum $t_0$. If $G$ exhibits Pair-LPST between $\textbf{e}^n_a - \textbf{e}^n_b$ and $\textbf{e}^n_c - \textbf{e}^n_d$ at $t_1$, then Lemma~\ref{fixed state}, Theorem~\ref{Conditions for the existence of perfect pair state transfer} and Theorem~\ref{More conditions for the existence of perfect pair state transfer} altogether imply that $t_1$ is a rational multiple of $t_0$. 

In the following, we state a theorem on simultaneous approximation of numbers, called \textit{Kronecker's approximation theorem}.
\begin{theorem}\emph{\cite{Apostol1990}}\label{Kronecker approximation theorem}
Let $\alpha_1, \ldots, \alpha_r$ be arbitrary real numbers. Also, let $1, y_1, \ldots, y_r$ be real and linearly independent over $\mathbb{Q}$. Then for every positive real number $\varepsilon$, there exist integers $\ell,s_1, \ldots, s_r$ such that 
$$|\ell y_j - s_j -\alpha_j |<\varepsilon~~\mathrm{for}~1 \leq j \leq r.$$
\end{theorem}

The next lemma provides some sets of numbers that are linearly independent over $\mathbb{Q}$.
\begin{lema}\emph{\cite{Richard1974}}\label{Linearly independent corollary}
The set $\{\sqrt{\Delta} \colon \Delta~\mathrm{is~a~square~free~integer}\}$ is linearly independent over $\mathbb{Q}$.
\end{lema}
Let $\alpha,\beta$ be complex numbers and $\varepsilon$ be an arbitrarily small positive real number. Then we write the inequality $|\alpha - \beta| < \varepsilon$ as $\alpha \approx \beta$ and omit the dependence on $\varepsilon$. In~\cite{Laplacianprettygoodedgestatetransferinpaths2022}, Wang et al. provided necessary and sufficient conditions for the existence of Pair-LPGST on graphs. 
\begin{theorem}\emph{\cite{Laplacianprettygoodedgestatetransferinpaths2022}}\label{Paths}
Let $\textbf{e}^n_a - \textbf{e}^n_b$ and $\textbf{e}^n_c - \textbf{e}^n_d$ be two linearly independent pair states of a graph $G$, $\Phi^{+}_{ab,cd}=\{\lambda_1,\ldots,\lambda_d\}$ and $\Phi^{-}_{ab,cd}=\{\mu_1,\ldots,\mu_q\}$. Then $G$ exhibits Pair-LPGST between $\textbf{e}^n_a - \textbf{e}^n_b$ and $\textbf{e}^n_c - \textbf{e}^n_d$  if and only if the following two conditions hold.
\begin{enumerate}
\item[(i)] The pair states $\textbf{e}^n_a - \textbf{e}^n_b$ and $\textbf{e}^n_c - \textbf{e}^n_d$ are Laplacian strongly cospectral.
\item[(ii)] If $\ell_1,\ldots,\ell_d,m_1,\ldots,m_q$ are integers such that
$$\sum_{i=1}^{d}\ell_i \lambda_i + \sum_{j=1}^{q}m_j \mu_j = 0~~\mathrm{and}~~\sum_{i=1}^{d}\ell_i + \sum_{j=1}^{q}m_j = 0,~~\mathrm{then}~\sum_{j=1}^{q}m_j~\mathrm{is~even}.$$
\end{enumerate}
\end{theorem}
In Section~\ref{pretty good}, we use the preceding theorem to study Pair-LPGST on the total graph of complete graphs. Let $G$ denote an $r$-regular graph, where $r \geq 2$. Let $\{v_1, \ldots, v_n\}$ be the  vertex set of $G$ and $\{e_1, \ldots, e_m\}$ be the edge set of $G$. The \textit{incidence matrix} of $G$ is denoted by $R_G:=(r_{ij})_{n \times m}$, where
\[r_{ij} = \left\{ \begin{array}{rl} 1 & \textrm{ if } v_i \textrm{ is incident to } e_j \\
0 & \textrm{ otherwise. } \end{array}\right. \]  
If $\{\textbf{z}_1,\ldots,\textbf{z}_s\}$ is an orthonormal basis of the null-space of $R_G$, then Lemma 2.17 of Bapat~\cite{Bapat2010} implies that 
\[s = \left\{ \begin{array}{rl} m-n & \textrm{ if G is non-bipartite } \\
m-n+1 & \textrm{ if G is bipartite. } \end{array}\right. \]
We consider the ordering $v_1, \ldots, v_n,e_1, \ldots, e_m$ of the vertices of $\mathcal{T}(G)$. Liu and Wang~\cite{TotalGraph2021} determined the Laplacian eigenvalues and the Laplacian eigenprojectors of $\mathcal{T}(G)$. We state it in the form of the following theorem.  
\begin{theorem}\emph{\cite{TotalGraph2021}}\label{Total graph spectra}
Let $\theta_0,\theta_1, \ldots, \theta_p$ be the Laplacian eigenvalues of $G$ with $0=\theta_0<\theta_1 < \cdots < \theta_p$ and $E_{\theta_0}{(G)},E_{\theta_1}{(G)}, \ldots, E_{\theta_p}{(G)}$ be the corresponding Laplacian eigenprojectors of $G$.
Also, let
\begin{align*}
&\theta^{\pm }_j = \frac{r+2+2\theta_j \pm \sqrt{(r+2)^2-4\theta_j}}{2}~~\mathrm{and}\\
&E_{\theta^{\pm }_j}{(\mathcal{T}(G))} = \frac{1}{(2+\theta_j-\theta^{\pm }_j)^2+2r-\theta_j}\begin{pmatrix}
         (2+\theta_j-\theta^{\pm }_j)^2E_{\theta_j}(G) & (2+\theta_j-\theta^{\pm }_j)E_{\theta_j}(G)R_G  \\
         (2+\theta_j-\theta^{\pm }_j){R^T_G}E_{\theta_j}(G)& {R^T_G}E_{\theta_j}(G)R_G  
       \end{pmatrix},
\end{align*}
where $0 \leq j \leq p$.  
\begin{enumerate}
\item[(i)] If $G$ is non-bipartite, then $\theta^{\pm }_j$ are Laplacian eigenvalues of $\mathcal{T}(G)$ with corresponding Laplacian eigenprojectors $E_{\theta^{\pm }_j}{(\mathcal{T}(G))}$ for $0 \leq j \leq p$. Further, $2r+2$ is a Laplacian eigenvalue of $\mathcal{T}(G)$ with corresponding Laplacian eigenprojector 
\begin{align*}
E_{2r+2}{(\mathcal{T}(G))}=\begin{pmatrix}
         \textbf{0}_{n \times n} & \textbf{0}_{n \times m}  \\
         \textbf{0}_{m \times n} & \sum_{i=1}^{m-n} \textbf{z}_i{\textbf{z}}^T_i  
       \end{pmatrix}.
\end{align*}
\item[(ii)] If $G$ is bipartite and $X_1 \cup X_2$ is a bipartition of the vertex set of $G$, then $\theta^{\pm }_j$ are Laplacian eigenvalues of $\mathcal{T}(G)$ with corresponding Laplacian eigenprojectors $E_{\theta^{\pm }_j}{(\mathcal{T}(G))}$ for $0 \leq j \leq p-1$ and $3r$ is a Laplacian eigenvalue of $\mathcal{T}(G)$ with corresponding Laplacian eigenprojector 
\begin{align*}
E_{3r}{(\mathcal{T}(G))}=\begin{pmatrix}
         Q & \textbf{0}_{n \times m}  \\
         \textbf{0}_{m \times n} & \textbf{0}_{m \times m}  
       \end{pmatrix},~\mathrm{where}~
Q=\frac{1}{n}\begin{pmatrix}
        J_{|X_1| \times |X_1|}  &-J_{|X_1| \times |X_2|}   \\
         -J_{|X_2| \times |X_1|} & J_{|X_2| \times |X_2|}  
       \end{pmatrix}.
\end{align*}
Further, $2r+2$ is a Laplacian eigenvalue of $\mathcal{T}(G)$ with corresponding Laplacian eigenprojector 
\begin{align*}
E_{2r+2}{(\mathcal{T}(G))}=\begin{pmatrix}
         \textbf{0}_{n \times n} & \textbf{0}_{n \times m}  \\
         \textbf{0}_{m \times n} & \sum_{i=1}^{m-n+1} \textbf{z}_i{\textbf{z}}^T_i  
       \end{pmatrix}.
\end{align*}
\end{enumerate}
\end{theorem}
     
The \textit{Cartesian product} of two graphs $G_1$ and $G_2$, denoted $G_1 \square G_2$, is defined as the graph whose vertex set is $V(G_1) \times V(G_2)$ such that two vertices $(g_1,h_1)$ and $(g_2,h_2)$ are adjacent in $G_1 \square G_2$ if exactly one of the following two conditions hold:
\begin{enumerate}
\item[(i)] $g_1=g_2$ and $h_1$ is adjacent to $h_2$ in $G_2$,
\item[(ii)] $g_1$ is adjacent to $g_2$ in $G_1$ and $h_1=h_2$.
\end{enumerate}
The \textit{tensor product} of two graphs $G_1$ and $G_2$, denoted $G_1 \times G_2$, is defined as the graph whose vertex set is $V(G_1) \times V(G_2)$ such that two vertices $(g_1,h_1)$ and $(g_2,h_2)$ are adjacent in $G_1 \times G_2$ if $g_1$ is adjacent to $g_2$ in $G_1$ and $h_1$ is adjacent to $h_2$ in $G_2$. Now we state some results that we use in Section~\ref{Concrete examples} to construct families of total graphs exhibiting Pair-LPGST.

\begin{theorem}\emph{\cite{CayleyGraphProduct2016}}\label{Cayley Graph Cartesian Product}
Let $G_1$ and $G_2$ be two Cayley graphs with $G_1=\mathrm{Cay}(\Gamma_1,S_1)$ and $G_2=\mathrm{Cay}(\Gamma_2,S_2)$. Then $G_1 \square G_2$ is also a Cayley graph given by $G_1 \square G_2 = \mathrm{Cay}(\Gamma_1 \times \Gamma_2, (S_1 \times \textbf{1}) \cup ({\textbf{1} \times S_2}))$.
\end{theorem}

\begin{theorem}\emph{\cite{CayleyGraphProduct2016}}\label{Cayley Graph Tensor Product}
Let $G_1$ and $G_2$ be two Cayley graphs with $G_1=\mathrm{Cay}(\Gamma_1,S_1)$ and $G_2=\mathrm{Cay}(\Gamma_2,S_2)$. Then $G_1 \times G_2$ is also a Cayley graph given by $G_1 \times G_2 = \mathrm{Cay}(\Gamma_1 \times \Gamma_2, S_1 \times S_2)$. 
\end{theorem}

\begin{theorem}\emph{\cite{Canul2010}}\label{PST and Cartesian product}
If the graphs $G_1$ and $G_2$ exhibit PST at $t$, then $G_1 \square G_2$ also exhibits PST at $t$.
\end{theorem}

\begin{theorem}\emph{\cite{GraphProduct2011}}\label{PST and tensor product}
Let $G_1$ be a graph exhibiting PST at $t$ and $t\lambda$ is an integer multiple of $\pi$, for all eigenvalues $\lambda$ of $G_1$. Also, let $G_2$ be a circulant graph with odd eigenvalues. Then the graph $G_1 \times G_2$ exhibits PST at $t$.
\end{theorem}

\begin{theorem}\emph{\cite{Coutinho2015}}\label{Distance regular graphs}
The cocktail party graph $\overline{\cup_m K_2}$ exhibits PST at $\frac{\pi}{2}$, where
$m$ is an even integer.
\end{theorem}

\begin{theorem}\emph{\cite{GaojunLuo2022}}\label{PST implies Pair PST}
If a Cayley graph exhibits LPST at $t$, then it also exhibits Pair-LPST at $t$.
\end{theorem}

\section{Laplacian perfect pair state transfer on total graphs}\label{perfect}
In this section, we explore Pair-LPST on total graphs of regular graphs. Let $G$ be an $r$-regular connected graph. If $r=1$, then $G$ is the path on two vertices and so $\mathcal{T}(G)$ is the cycle on three vertices. Therefore Theorem 7.3 of Chen and Godsil~\cite{QChen2020} implies that the graph $\mathcal{T}(G)$ does not exhibit Pair-LPST. In what follows, $G$ denotes an $r$-regular connected graph on $n$ vertices and $m$ edges such that $r \geq 3$. Recall that if $\{v_1, \ldots, v_n\}$ is the vertex set of $G$ and $\{e_1, \ldots, e_m\}$ is the edge set of $G$, then we consider the ordering $v_1, \ldots, v_n,e_1, \ldots, e_m$ of the vertices of the graph $\mathcal{T}(G)$. Let $a$ be a vertex of $\mathcal{T}(G)$ such that $a$ is a vertex of $G$. Then the vertex state of $a$ in $\mathcal{T}(G)$, denoted $\textbf{e}^{n+m}_a$, is the column vector
$$\begin{pmatrix}
     \textbf{e}^n_a \\
     \textbf{0}_m  
\end{pmatrix}.$$
Let $a$ be a vertex of $\mathcal{T}(G)$ such that $a$ is an edge of $G$. Then the vertex state of $a$ in $\mathcal{T}(G)$, denoted $\textbf{e}^{n+m}_a$, is the column vector
$$\begin{pmatrix}
     \textbf{0}_n \\
     \textbf{e}^m_a 
\end{pmatrix}.$$
Using some ideas from Jiang et al.~\cite{PairStateTransferQGraph2025}, we prove the following lemma. In what follows, $\Phi_{ab}$ denotes the Laplacian eigenvalue support of the pair state $\textbf{e}^{n+m}_a - \textbf{e}^{n+m}_b$ of the graph $\mathcal{T}(G)$.
\begin{lema}\label{perfect state transfer lemma non-bipartite}
Let $G$ be an $r$-regular connected non-bipartite graph such that $r \geq 3$. If $a$ and $b$ are vertices of $G$; or $a$ and $b$ are edges of $G$, then the following two conditions hold.
\begin{enumerate}
\item[(i)] $\theta^{+}_j \in \Phi_{ab}$ if and only if $\theta^{-}_j \in \Phi_{ab}$ for $0 \leq j \leq p$.
\item[(ii)] Let $\textbf{e}^{n+m}_a - \textbf{e}^{n+m}_b$ and $\textbf{e}^{n+m}_c - \textbf{e}^{n+m}_d$ be Laplacian strongly cospectral pair states of $\mathcal{T}(G)$. Then 
$\theta^{+}_j \in \Phi^{+}_{ab,cd}$ if and only if $\theta^{-}_j \in \Phi^{+}_{ab,cd}$ for $0 \leq j \leq p$. Further, $\theta^{+}_j \in \Phi^{-}_{ab,cd}$ if and only if $\theta^{-}_j \in \Phi^{-}_{ab,cd}$ for $0 \leq j \leq p$.
\end{enumerate} 
\end{lema}
\begin{proof}
$(i)$ If $a$ and $b$ are vertices of $G$, then Theorem~\ref{Total graph spectra} implies that 
\begin{equation}\label{x_1}
E_{\theta^{\pm }_j}(\mathcal{T}(G))(\textbf{e}^{n+m}_a - \textbf{e}^{n+m}_b) = \frac{1}{(2+\theta_j-\theta^{\pm }_j)^2+2r-\theta_j}\begin{pmatrix}
         (2+\theta_j-\theta^{\pm }_j)^2E_{\theta_j}(G)(\textbf{e}^n_a - \textbf{e}^n_b)  \\
         (2+\theta_j-\theta^{\pm }_j){R^T_G}E_{\theta_j}(G)(\textbf{e}^n_a - \textbf{e}^n_b)  
       \end{pmatrix}.
\end{equation}  
If possible, let $2+\theta_j-\theta^{\pm }_j = 0$ for some $j$ such that $0 \leq j \leq p$. Then it follows that $\theta_j = 2r$, which contradicts the fact that $G$ is non-bipartite. Thus $2+\theta_j-\theta^{\pm }_j \neq 0$ for $0 \leq j \leq p$. From Equation~(\ref{x_1}), we find that $E_{\theta^{+}_j}(\mathcal{T}(G))(\textbf{e}^{n+m}_a - \textbf{e}^{n+m}_b) \neq \textbf{0}_{n+m}$ if and only if $E_{\theta^{-}_j}(\mathcal{T}(G))(\textbf{e}^{n+m}_a - \textbf{e}^{n+m}_b) \neq \textbf{0}_{n+m}$. Therefore $\theta^{+}_j \in \Phi_{ab}$ if and only if $\theta^{-}_j \in \Phi_{ab}$ for $0 \leq j \leq p$.

Similarly, if $a$ and $b$ are edges of $G$, then 
\begin{equation}\label{y_1}
E_{\theta^{\pm }_j}(\mathcal{T}(G))(\textbf{e}^{n+m}_a - \textbf{e}^{n+m}_b) = \frac{1}{(2+\theta_j-\theta^{\pm }_j)^2+2r-\theta_j}\begin{pmatrix}
         (2+\theta_j-\theta^{\pm }_j)E_{\theta_j}(G)R_G(\textbf{e}^m_a - \textbf{e}^m_b)  \\
         {R^T_G}E_{\theta_j}(G)R_G (\textbf{e}^m_a - \textbf{e}^m_b)
       \end{pmatrix}.
\end{equation}
From Equation~(\ref{y_1}), we find that $\theta^{+}_j \in \Phi_{ab}$ if and only if $\theta^{-}_j \in \Phi_{ab}$ for $0 \leq j \leq p$.

$(ii)$ Suppose that $\textbf{e}^{n+m}_a - \textbf{e}^{n+m}_b$ and $\textbf{e}^{n+m}_c - \textbf{e}^{n+m}_d$ are Laplacian strongly cospectral. First consider the case that $a$ and $b$ are vertices of $G$. Then we have 
 \begin{equation}\label{y_2}
E_{\theta^{\pm }_j}(\mathcal{T}(G))(\textbf{e}^{n+m}_a - \textbf{e}^{n+m}_b) = \pm E_{\theta^{\pm }_j}(\mathcal{T}(G))(\textbf{e}^{n+m}_c - \textbf{e}^{n+m}_d)~~~\mathrm{for}~0 \leq j \leq p.
\end{equation}
Since $G$ is a connected graph on $n$ vertices, we have $E_{\theta_0}(G) = \frac{1}{n}J_{n \times n}$. This implies $E_{\theta_0}(G)(\textbf{e}^n_a - \textbf{e}^n_b)=\textbf{0}_n$. Also, we have $\theta^{+}_0=r+2$. Now Equation~(\ref{x_1}) and Equation~(\ref{y_2}) yield that 
\begin{equation}\label{x_2}
E_{r+2}(\mathcal{T}(G))(\textbf{e}^{n+m}_a - \textbf{e}^{n+m}_b) = \textbf{0}_{n+m} = E_{r+2}(\mathcal{T}(G))(\textbf{e}^{n+m}_c - \textbf{e}^{n+m}_d).
\end{equation}
Using Theorem~\ref{Total graph spectra}, we obtain 
\begin{equation}\label{x_3}
E_{r+2}(\mathcal{T}(G)) = \frac{1}{r^2+2r}\begin{pmatrix}
         \frac{r^2}{n} J_{n\times n} &  \frac{-2r}{n} J_{n \times m}\\
          \frac{-2r}{n}J_{m \times n} &  \frac{4}{n} J_{m \times m}
       \end{pmatrix}.
\end{equation}
Now Equation~(\ref{x_2}) and Equation~(\ref{x_3}) together imply that either $c,d$ are vertices of $G$; or $c,d$ are edges of $G$. If possible, let $c, d$ be edges of the graph 
$G$. Then we have
\begin{equation}\label{x_4}
E_{\theta^{\pm }_j}(\mathcal{T}(G))(\textbf{e}^{n+m}_c - \textbf{e}^{n+m}_d) = \frac{1}{(2+\theta_j-\theta^{\pm }_j)^2+2r-\theta_j}\begin{pmatrix}
         (2+\theta_j-\theta^{\pm }_j)E_{\theta_j}(G)R_G(\textbf{e}^m_c - \textbf{e}^m_d)  \\
         {R^T_G}E_{\theta_j}(G)R_G(\textbf{e}^m_c - \textbf{e}^m_d)  
       \end{pmatrix}.
\end{equation}      
Suppose that there exists $j$ with $0 \leq j \leq p$ and $\theta^{\pm}_j \in \Phi_{ab}$. Consider the following two cases. 

\noindent \textbf{Case 1.} Let $\theta^{+}_j \in \Phi^{+}_{ab,cd}$ and $\theta^{-}_j \in \Phi^{-}_{ab,cd}$. Then we have  
\begin{align}
&E_{\theta^{+ }_j}(\mathcal{T}(G))(\textbf{e}^{n+m}_a - \textbf{e}^{n+m}_b) = E_{\theta^{+}_j}(\mathcal{T}(G))(\textbf{e}^{n+m}_c - \textbf{e}^{n+m}_d)~\mathrm{and}~\label{T_1}\\
&E_{\theta^{-}_j}(\mathcal{T}(G))(\textbf{e}^{n+m}_a - \textbf{e}^{n+m}_b) = -E_{\theta^{-}_j}(\mathcal{T}(G))(\textbf{e}^{n+m}_c - \textbf{e}^{n+m}_d)\label{T_2}.
\end{align}
Using equations~(\ref{x_1}) and~(\ref{x_4}) in equations~(\ref{T_1}) and~(\ref{T_2}), we obtain  
\begin{align*}
&E_{\theta_j}(G)R_G(\textbf{e}^m_c - \textbf{e}^m_d) = (2+\theta_j-\theta^{+}_j)E_{\theta_j}(G)(\textbf{e}^n_a - \textbf{e}^n_b)~\mathrm{and}~\\
&E_{\theta_j}(G)R_G(\textbf{e}^m_c - \textbf{e}^m_d) = -(2+\theta_j-\theta^{-}_j)E_{\theta_j}(G)(\textbf{e}^n_a - \textbf{e}^n_b).
\end{align*}
As $E_{\theta_j}(G)(\textbf{e}^n_a - \textbf{e}^n_b) \neq \textbf{0}_n$, we have from the preceding two equations that $\theta^{+}_j + \theta^{-}_j = 2\theta_j + 4$. This implies $r=2$, which contradicts the assumption that $r \geq 3$. We get the same contradiction if  $\theta^{-}_j \in \Phi^{+}_{ab,cd}$ and $\theta^{+}_j \in \Phi^{-}_{ab,cd}$.

\noindent \textbf{Case 2.} Let $\theta^{+}_j,\theta^{-}_j \in \Phi^{+}_{ab,cd}$ or $\theta^{+}_j, \theta^{-}_j \in \Phi^{-}_{ab,cd}$. In this case also, proceeding as in Case 1 we find $\theta^{+}_j = \theta^{-}_j$. This implies $(r+2)^2-4\theta_j=0$. This is a contradiction as $0 < (r-2)^2 <(r+2)^2-4\theta_j$ for $0 \leq j \leq p$. 

Thus $\theta^{\pm}_j \notin \Phi_{ab}$ for $0 \leq j \leq p$. However, By Lemma~\ref{fixed state}, $|\Phi_{ab}| \geq 2$ and by Theorem~\ref{Total graph spectra}, the Laplacian eigenvalues of $\mathcal{T}(G)$ are $\theta^{\pm}_j$ and $2r+2$ for $0 \leq j \leq p$. So,  $\Phi_{ab}$ must contain at least one of $\theta^{+}_j$ or $\theta^{-}_j$ for some $j$, which is not possible. Hence, $c$ and $d$ must be vertices of the graph $G$. Therefore, 
\begin{equation}\label{x_5}
E_{\theta^{\pm }_j}(\mathcal{T}(G))(\textbf{e}^{n+m}_c - \textbf{e}^{n+m}_d) = \frac{1}{(2+\theta_j-\theta^{\pm }_j)^2+2r-\theta_j}\begin{pmatrix}
         (2+\theta_j-\theta^{\pm }_j)^2E_{\theta_j}(G)(\textbf{e}^n_c - \textbf{e}^n_d)  \\
         (2+\theta_j-\theta^{\pm }_j){R^T_G}E_{\theta_j}(G)(\textbf{e}^n_c - \textbf{e}^n_d)  
       \end{pmatrix}.
\end{equation}          
From equations (\ref{x_1}) and (\ref{x_5}), we find  $\theta^{+}_j \in \Phi^{+}_{ab,cd}$ if and only if $E_{\theta_j}(G)(\textbf{e}^n_a - \textbf{e}^n_b) = E_{\theta_j}(G)(\textbf{e}^n_c - \textbf{e}^n_d) \neq \textbf{0}_n$. Similarly, $\theta^{-}_j \in \Phi^{+}_{ab,cd}$ if and only if $E_{\theta_j}(G)(\textbf{e}^n_a - \textbf{e}^n_b) = E_{\theta_j}(G)(\textbf{e}^n_c - \textbf{e}^n_d) \neq \textbf{0}_n$. Hence, $\theta^{+}_j \in \Phi^{+}_{ab,cd}$ if and only if $\theta^{-}_j \in \Phi^{+}_{ab,cd}$ for $0 \leq j \leq p$. In a similar way, $\theta^{+}_j \in \Phi^{-}_{ab,cd}$ if and only if $\theta^{-}_j \in \Phi^{-}_{ab,cd}$ for $0 \leq j \leq p$.    

The proof for the case that $a$ and $b$ are edges of $G$ is similar to the preceding part, 
and hence the details are omitted. 
\end{proof}

The proof of the next lemma is similar to that of Lemma~\ref{perfect state transfer lemma non-bipartite}. 
\begin{lema}\label{perfect state transfer lemma bipartite}
Let $G$ be an $r$-regular connected bipartite graph such that $r \geq 3$. If $a$ and $b$ are vertices of $G$; or $a$ and $b$ are edges of $G$, then the following two conditions hold.
\begin{enumerate}
\item[(i)] $\theta^{+}_j \in \Phi_{ab}$ if and only if $\theta^{-}_j \in \Phi_{ab}$ for $0 \leq j \leq p-1$.
\item[(ii)] Let $\textbf{e}^{n+m}_a - \textbf{e}^{n+m}_b$ and $\textbf{e}^{n+m}_c - \textbf{e}^{n+m}_d$ be Laplacian strongly cospectral. Then 
$\theta^{+}_j \in \Phi^{+}_{ab,cd}$ if and only if $\theta^{-}_j \in \Phi^{+}_{ab,cd}$ for $0 \leq j \leq p-1$. Further, $\theta^{+}_j \in \Phi^{-}_{ab,cd}$ if and only if $\theta^{-}_j \in \Phi^{-}_{ab,cd}$ for $0 \leq j \leq p-1$.
\end{enumerate} 
\end{lema}

Now we use Lemma~\ref{perfect state transfer lemma non-bipartite} and Lemma~\ref{perfect state transfer lemma bipartite} to prove the next four lemmas. 
\begin{lema}\label{non-existence of PST 1st Lemma}
Let $a$ and $b$ be two vertices of an $r$-regular connected non-bipartite graph $G$ such that $r \geq 3$. Also, let $r+1$ not be a Laplacian eigenvalue of $G$. Then the pair state $\textbf{e}^{n+m}_a - \textbf{e}^{n+m}_b$ of $\mathcal{T}(G)$ cannot be involved in Pair-LPST.  
\end{lema}
\begin{proof}
Let $\mathcal{T}(G)$ exhibit Pair-LPST between the pair states $\textbf{e}^{n+m}_a - \textbf{e}^{n+m}_b$ and $\textbf{e}^{n+m}_c - \textbf{e}^{n+m}_d$. Since $a$ and $b$ are vertices of $G$, from Theorem~\ref{Total graph spectra} and Equation~(\ref{x_3}) we have 
$$E_{\theta^{\pm }_0}{(\mathcal{T}(G))}(\textbf{e}^{n+m}_a - \textbf{e}^{n+m}_b) = \textbf{0}_{n+m}= E_{2r+2}{(\mathcal{T}(G))}(\textbf{e}^{n+m}_a - \textbf{e}^{n+m}_b).$$ 
This implies that $\Phi_{ab} \subseteq \{\theta^{\pm }_j \colon 1 \leq j \leq p\}$. By Theorem~\ref{Conditions for the existence of perfect pair state transfer}, either all the elements in $\Phi_{ab}$ are integers or all the elements in $\Phi_{ab}$ are quadratic integers. From Lemma~\ref{perfect state transfer lemma non-bipartite}, for each $j$ with $1 \leq j \leq p$ we have either $\theta^{\pm }_j \in \Phi_{ab}$ or $\theta^{\pm }_j \notin \Phi_{ab}$.

Suppose that $\theta^{\pm }_j$ is an integer and $\theta^{\pm }_j \in \Phi_{ab}$ for some $j$ with $1 \leq j \leq p$. Then $\theta_j$ is an integer and $(r+2)^2-4\theta_j$ is a perfect square. Since $0 < \theta_j < 2r$, we have $\sqrt{(r+2)^2-4\theta_j} \in \{r-1,r,r+1\}$. If $\sqrt{(r+2)^2-4\theta_j} = r-1$ or $\sqrt{(r+2)^2-4\theta_j} = r+1$, then $\theta^{+}_j \notin \mathbb{Z}$, which is a contradiction. If $\sqrt{(r+2)^2-4\theta_j} = r$, then $\theta_j = r+1$, which is again a contradiction. Thus we find that $\theta^{\pm }_j \notin \Phi_{ab}$ for $1 \leq j \leq p$. This implies that $\Phi_{ab}$ is an empty set, which is a contradiction.

Suppose that $\theta^{\pm }_j$ is a quadratic integer and $\theta^{\pm }_j \in \Phi_{ab}$ for some $j$ with $1 \leq j \leq p$. Then Theorem~\ref{Conditions for the existence of perfect pair state transfer} yields that $\theta^{\pm }_j=\frac{1}{2}(x+y_{\pm}\sqrt{\Delta})$, where $\Delta$ is a square-free integer such that $\Delta>1$, $x,y_{\pm} \in \mathbb{Z}$ and, $y_{+}$ and $y_{-}$ have the same parity. This, along with Lemma~\ref{Total graph spectra}, give 
\begin{equation}\label{b_1}
x + \frac{1}{2}(y_{+} + y_{-})\sqrt{\Delta} = \theta^{+}_j + \theta^{-}_j = r+2+2\theta_j.
\end{equation}
This implies that
\begin{equation}\label{x_7}
4\theta_j = 2(x - r - 2) + (y_{+} + y_{-})\sqrt{\Delta}.
\end{equation}
From Theorem~\ref{Total graph spectra} and Equation~(\ref{x_7}), we obtain that 
\begin{equation}\label{x_8}
4\theta^{+}_j\theta^{-}_j = (x-r-2)(x+r+4)+\frac{(y_{+}+y_{-})^2}{4}\Delta + (x+1)(y_{+}+y_{-})\sqrt{\Delta}.
\end{equation}
Also, we have
\begin{equation}\label{x_9}
4\theta^{+}_j\theta^{-}_j = x^2 + y_{+}y_{-}\Delta + x(y_{+}+y_{-})\sqrt{\Delta}.
\end{equation}
As $\sqrt{\Delta}$ is an irrational number, Equation~(\ref{x_8}) and Equation~(\ref{x_9}) yield that 
\begin{equation}\label{b_2}
(x+1)(y_{+}+y_{-}) = x(y_{+}+y_{-}).
\end{equation}
From Equation~(\ref{b_2}) we find that $y_{+}+y_{-} = 0$. Then Equation~(\ref{x_7}) gives $\theta_j=\frac{x - r - 2}{2}$, from which we find that $\Phi_{ab} = \left\{{\theta_j}^{+},{\theta_j}^{-}\right\}$. Therefore by Lemma~\ref{perfect state transfer lemma non-bipartite}, we have 
$\Phi^{+}_{ab,cd} = \{{\theta_j}^{+},{\theta_j}^{-}\}$ or $\Phi^{-}_{ab,cd} = \{{\theta_j}^{+},{\theta_j}^{-}\}$.
This implies that either $\Phi^{-}_{ab,cd}$ is empty or $\Phi^{+}_{ab,cd}$ is empty. This contradicts Theorem~\ref{fixed state}. Hence, the pair state $\textbf{e}^{n+m}_a - \textbf{e}^{n+m}_b$ cannot be involved in Pair-LPST.
\end{proof}

\begin{lema}\label{non-existence of PST 2nd Lemma}
Let $a$ and $b$ be two vertices of an $r$-regular connected bipartite graph $G$ such that $r \geq 3$. Also, let $r+1$ not be a Laplacian eigenvalue of $G$. Then the pair state $\textbf{e}^{n+m}_a - \textbf{e}^{n+m}_b$ of $\mathcal{T}(G)$ cannot be involved in Pair-LPST.   
\end{lema}
\begin{proof}
Let $\mathcal{T}(G)$ exhibit Pair-LPST between the pair states $\textbf{e}^{n+m}_a - \textbf{e}^{n+m}_b$ and $\textbf{e}^{n+m}_c - \textbf{e}^{n+m}_d$. We have $\Phi_{ab} \subseteq \{\theta^{\pm }_j \colon 1 \leq j \leq p-1\} \cup \{3r\}$. Let $X_1 \cup X_2$ be a bipartition of the vertex set of $G$. Then the following two cases arise.

\noindent \textbf{Case 1.} Let $a, b \in X_1$ or $a, b \in X_2$. Then by Theorem~\ref{Total graph spectra}, we find that $3r \notin \Phi_{ab}$. This implies that $\Phi_{ab} \subseteq \{\theta^{\pm }_j \colon 1 \leq j \leq p-1\}$. Then following the same procedure as in Lemma~\ref{non-existence of PST 1st Lemma}, we arrive at a contradiction. 

\noindent \textbf{Case 2.} Let $a \in X_1,~b \in X_2$ or $b \in X_1,~a \in X_2$. Then by Theorem~\ref{Total graph spectra}, $3r \in \Phi_{ab}$. Since the pair state $\textbf{e}^{n+m}_a - \textbf{e}^{n+m}_b$ involves in Laplacian perfect pair state transfer, Theorem~\ref{Conditions for the existence of perfect pair state transfer} implies that all the eigenvalues of $\Phi_{ab}$ are integers. Now one can prove that $\theta^{\pm }_j \notin \Phi_{ab}$ for $1 \leq j \leq p-1$. Thus $\Phi_{ab} = \{3r\}$, contradicting Theorem~\ref{fixed state}. Hence the pair state $\textbf{e}^{n+m}_a - \textbf{e}^{n+m}_b$ of $\mathcal{T}(G)$ cannot be involved in Pair-LPST. 
\end{proof}

\begin{lema}\label{non-existence of PST 3rd Lemma}
Let $a$ and $b$ be two edges of an $r$-regular connected non-bipartite graph $G$ such that $r \geq 3$. Also, let $r+1$ not be a Laplacian eigenvalue of $G$. Then the pair state $\textbf{e}^{n+m}_a - \textbf{e}^{n+m}_b$ of $\mathcal{T}(G)$ cannot be involved in Pair-LPST.  
\end{lema}
\begin{proof}
Let $\mathcal{T}(G)$ exhibit Pair-LPST between $\textbf{e}^{n+m}_a - \textbf{e}^{n+m}_b$ and $\textbf{e}^{n+m}_c - \textbf{e}^{n+m}_d$. Since $a$ and $b$ are two edges of $G$, we have 
$$E_{\theta^{\pm }_0}{(\mathcal{T}(G))}(\textbf{e}^{n+m}_a - \textbf{e}^{n+m}_b) = \textbf{0}_{n+m}.$$
This implies that $\Phi_{ab} \subseteq \{\theta^{\pm }_j \colon 1 \leq j \leq p\} \cup \{2r+2\}$. Now the rest of the proof is similar to that of Lemma~\ref{non-existence of PST 1st Lemma}.
\end{proof}

\begin{lema}\label{non-existence of PST 4th Lemma}
Let $a$ and $b$ be two edges of an $r$-regular connected bipartite graph $G$ such that $r \geq 3$. Also, let $r+1$ not be a Laplacian eigenvalue of $G$. Then the pair state $\textbf{e}^{n+m}_a - \textbf{e}^{n+m}_b$ of $\mathcal{T}(G)$ cannot be involved in Pair-LPST.  
\end{lema}
\begin{proof}
Let $\mathcal{T}(G)$ exhibit Pair-LPST between $\textbf{e}^{n+m}_a - \textbf{e}^{n+m}_b$ and $\textbf{e}^{n+m}_c - \textbf{e}^{n+m}_d$. Since $a$ and $b$ are two edges of $G$, we obtain from Theorem~\ref{Total graph spectra} and Equation~(\ref{x_3}) that $\theta^{+}_0,\theta^{-}_0,3r \notin \Phi_{ab}$. This implies that  $\Phi_{ab} \subseteq \{\theta^{\pm }_j \colon 1 \leq j \leq p-1\} \cup \{2r+2\}$. 
Now the rest of the proof is similar to that of Lemma~\ref{non-existence of PST 1st Lemma}.
\end{proof}

Let $G$ be an $r$-regular connected graph with $r \geq 3$. Also, let $a$ be a vertex and $b$ be an edge of $G$. The next two lemmas tell us that if $G$ is Laplacian integral and $r+1$ is not a Laplacian eigenvalue of $G$, then the pair state $\textbf{e}^{n+m}_a - \textbf{e}^{n+m}_b$ of the total graph of  $G$ cannot be involved in Pair-LPST.  

\begin{lema}\label{non-existence of PST 5th Lemma}
Let $a$ be a vertex and $b$ be an edge of an $r$-regular Laplacian integral connected non-bipartite graph $G$ such that $r \geq 3$. Also, let $r+1$ not be a Laplacian eigenvalue of $G$. Then the pair state $\textbf{e}^{n+m}_a - \textbf{e}^{n+m}_b$ of $\mathcal{T}(G)$ cannot be involved in Pair-LPST. 
\end{lema}
\begin{proof}
Let $\mathcal{T}(G)$ exhibit Pair-LPST between $\textbf{e}^{n+m}_a - \textbf{e}^{n+m}_b$ and $\textbf{e}^{n+m}_c - \textbf{e}^{n+m}_d$. Since $a$ is a vertex and $b$ is an edge of $G$, from Equation~(\ref{x_3}) we have
$$E_{\theta^{+}_0}(\mathcal{T}(G))(\textbf{e}^{n+m}_a - \textbf{e}^{n+m}_b) \neq \textbf{0}_{n+m}.$$
This implies that $\theta^{+}_0 \in \Phi_{ab}$. Since $\theta^{+}_0=r+2$, an integer, and $\textbf{e}^{n+m}_a - \textbf{e}^{n+m}_b$ involves in Pair-LPST, Theorem~\ref{Conditions for the existence of perfect pair state transfer} yields that all the eigenvalues in $\Phi_{ab}$ are integers. If possible, let $\theta^{+}_j \in \Phi_{ab}$ for some $j$ with $1 \leq j \leq p$. From Theorem~\ref{Total graph spectra}, recall that 
\begin{equation}\label{a_1}
\theta^{\pm }_j=\frac{r+2+2\theta_j \pm \sqrt{(r+2)^2-4\theta_j}}{2}.
\end{equation} 
Thus $\theta^{+}_j + \theta^{-}_j = r+2+2 \theta_j$. Since  
$\theta_j$ is an integer, we obtain that $\theta^{-}_j$ is also an integer. Similarly, if we assume that $\theta^{-}_j \in \Phi_{ab}$ for some $j$ with $1 \leq j \leq p$, then $\theta^{+}_j$ is also an integer. Now proceeding as in the second paragraph in the proof of Lemma~\ref{non-existence of PST 1st Lemma}, we arrive at a contradiction. From Theorem~\ref{fixed state}, $\Phi_{ab}$ contains at least two distinct eigenvalues. Thus we have $\Phi_{ab} = \{r+2,2r+2\}$. Therefore, post-multiplying by $\textbf{e}^{n+m}_a - \textbf{e}^{n+m}_b$ on the both sides of 
$$E_{r+2}(\mathcal{T}(G))+E_{2r+2}(\mathcal{T}(G))+\sum_{j=1}^{p}E_{\theta^{+}_j}(\mathcal{T}(G)) + \sum_{j=0}^{p}E_{\theta^{-}_j}(\mathcal{T}(G)) =I_{(n+m) \times (n+m)},$$
we have
$$\left(E_{r+2}(\mathcal{T}(G))+E_{2r+2}(\mathcal{T}(G))\right)\left(\textbf{e}^{n+m}_a - \textbf{e}^{n+m}_b\right)=\textbf{e}^{n+m}_a - \textbf{e}^{n+m}_b.$$
This further implies 
$$E_{r+2}(\mathcal{T}(G))_{aa}-E_{r+2}(\mathcal{T}(G))_{ab}+E_{2r+2}(\mathcal{T}(G))_{aa}-E_{2r+2}(\mathcal{T}(G))_{ab} = 1.$$
Now it follows that $n =1$, which is a contradiction. Therefore, $\textbf{e}^{n+m}_a - \textbf{e}^{n+m}_b$ cannot be involved in Pair-LPST.  
\end{proof}

\begin{lema}\label{non-existence of PST 6th Lemma}
Let $a$ be a vertex and $b$ be an edge of an $r$-regular Laplacian integral connected bipartite graph $G$ such that $r \geq 3$. Also, let $r+1$ not be a Laplacian eigenvalue of $G$. Then the pair state $\textbf{e}^{n+m}_a - \textbf{e}^{n+m}_b$ of $\mathcal{T}(G)$ cannot be involved in Pair-LPST. 
\end{lema}
\begin{proof}
Let $\mathcal{T}(G)$ exhibit Pair-LPST between $\textbf{e}^{n+m}_a - \textbf{e}^{n+m}_b$ and $\textbf{e}^{n+m}_c - \textbf{e}^{n+m}_d$. Clearly, $r+2,3r \in \Phi_{ab}$. Therefore, all the eigenvalues in $\Phi_{ab}$ are integers. Now proceeding as in the proof of Lemma~\ref{non-existence of PST 5th Lemma}, we obtain that $\Phi_{ab} \subseteq \{r+2,2r+2,3r\}$. Thus 
$$\left(E_{r+2}(\mathcal{T}(G))+E_{2r+2}(\mathcal{T}(G))+E_{3r}(\mathcal{T}(G))\right)(\textbf{e}^{n+m}_a - \textbf{e}^{n+m}_b)=\textbf{e}^{n+m}_a - \textbf{e}^{n+m}_b.$$
This implies that $n=2$, which is a contradiction. Therefore, $\textbf{e}^{n+m}_a - \textbf{e}^{n+m}_b$ cannot be involved in Pair-LPST.  
\end{proof}
Let $a$ be a vertex and $b$ be an edge of $G$. If the graph $G$ is Laplacian non-integral, then it is not yet known whether the pair state $\textbf{e}^{n+m}_a - \textbf{e}^{n+m}_b$ of $\mathcal{T}(G)$ can be involved in Pair-LPST or not. Further, if $G$ is the cycle on $n$ vertices, then it is not yet known whether the total graph of $G$ exhibits Pair-LPST or not. On combining Lemma~\ref{non-existence of PST 1st Lemma} to Lemma~\ref{non-existence of PST 6th Lemma}, we obtain the following theorem.
\begin{theorem}\label{non-existence of PST main result}
Let $G$ be an $r$-regular Laplacian integral connected graph such that $r \geq 3$. Also, let $r+1$ not be a Laplacian eigenvalue of $G$. Then the total graph of $G$ does not exhibit Pair-LPST. 
\end{theorem}

\section{Laplacian pretty good pair state transfer on total graphs}\label{pretty good}
In this section, we prove that there is Pair-LPGST on total graphs of regular connected graphs, under some mild conditions. Recall that $G$ denotes an $r$-regular connected graph, where $r \geq 3$ and $\theta_0,\theta_1, \ldots, \theta_p$ are the distinct Laplacian eigenvalues of $G$ such that $0=\theta_0 < \theta_1 < \cdots < \theta_p$. Let $\Delta_j = \sqrt{(r + 2)^2 - 4 \theta_j}$, where $0 \leq j \leq p$ if $G$ is non-bipartite and $0 \leq j \leq p-1$ if $G$ is bipartite. In the following, we state a theorem of Liu and Wang~\cite{TotalGraph2021}.

\begin{theorem}\emph{\cite{TotalGraph2021}}\label{Total graph entry vertex state}
Let $a$ and $b$ be two vertices of an $r$-regular connected graph $G$ such that $r\geq3$. Then the following two conditions hold.
\begin{enumerate}
\item[(a)] If $G$ is non-bipartite, then
\begin{align*}
U_{\mathcal{T}(G)}(t)_{ab} =\sum_{j=0}^{p}\exp(-\textbf{i}t((r + 2\theta_j + 2)/2))E_{\theta_j}(G)_{ab}\left(\cos\left(\frac{\Delta_j t}{2}\right) + \textbf{i}\frac{2-r}{\Delta_j}\sin\left(\frac{\Delta_j t}{2}\right)\right).
\end{align*} 

\item[(b)] If $G$ is bipartite, then
\begin{align*}
&U_{\mathcal{T}(G)}(t)_{ab}
=\sum_{j=0}^{p-1}\exp(-\textbf{i}t((r+ 2\theta_j + 2)/2))E_{\theta_j}(G)_{ab}\left(\cos\left(\frac{\Delta_j t}{2}\right) + \textbf{i}\frac{2-r}{\Delta_j}\sin\left(\frac{\Delta_j t}{2}\right)\right)\\
&+ \exp(-\textbf{i}3tr)E_{2r}(G)_{ab}.
\end{align*}
\end{enumerate}
\end{theorem}

Using Theorem~\ref{Total graph entry vertex state}, we prove the following lemma. The proof is straightforward and hence details of the proof are omitted.
\begin{lema}\label{Total graph entry pair state}
Let $a,b,c$ and $d$ be vertices of an $r$-regular connected graph $G$ such that $r\geq3$. Then the following two conditions hold.
\begin{enumerate}
\item[(a)] If $G$ is non-bipartite, then
\begin{align*}
&\frac{1}{2}{\left(\textbf{e}^{n+m}_a - \textbf{e}^{n+m}_b\right)}^TU_{\mathcal{T}(G)}(t)\left(\textbf{e}^{n+m}_c - \textbf{e}^{n+m}_d\right)\\ 
=&\frac{1}{2}\sum_{j=0}^{p}
\exp(-\textbf{i}t((r+2\theta_j+2)/2))
{\left(\textbf{e}^n_{a} - \textbf{e}^n_{b}\right)}^TE_{\theta_j}(G)\left(\textbf{e}^n_{c} - \textbf{e}^n_{d}\right)\\
&\times \left(\cos\left(\frac{\Delta_j t}{2}\right) + \textbf{i}\frac{2-r}{\Delta_j}\sin\left(\frac{\Delta_j t}{2}\right)\right).
\end{align*}

\item[(b)] If $G$ is bipartite, then
\begin{align*}
&\frac{1}{2}{\left(\textbf{e}^{n+m}_a - \textbf{e}^{n+m}_b\right)}^TU_{\mathcal{T}(G)}(t)\left(\textbf{e}^{n+m}_c - \textbf{e}^{n+m}_d\right)\\ 
=&\frac{1}{2}\sum_{j=0}^{p-1}
\exp(-\textbf{i}t((r+2\theta_j+2)/2))
{\left(\textbf{e}^n_{a} - \textbf{e}^n_{b}\right)}^TE_{\theta_j}(G)\left(\textbf{e}^n_{c} - \textbf{e}^n_{d}\right)\\
&\times\left(\cos\left(\frac{\Delta_j t}{2}\right) + \textbf{i}\frac{2-r}{\Delta_j}\sin\left(\frac{\Delta_j t}{2}\right)\right) + \frac{1}{2} \exp(-\textbf{i}3tr){\left(\textbf{e}^n_{a} - \textbf{e}^n_{b}\right)}^TE_{2r}(G)\left(\textbf{e}^n_{c} - \textbf{e}^n_{d}\right).
\end{align*}
\end{enumerate}
\end{lema}
The following theorem provides sufficient conditions for the existence of Pair-LPGST on the total graph of a regular connected non-bipartite graph.
\begin{theorem}\label{First theorem PGST}
Let $G$ be an $r$-regular connected non-bipartite graph such that $r\geq3$. Also, let $G$ exhibit Pair-LPST between $\textbf{e}^n_a - \textbf{e}^n_b$ and $\textbf{e}^n_c - \textbf{e}^n_d$ at $\frac{\pi}{2}$. If $r+1$ is not a Laplacian eigenvalue of $G$ and $\frac{r+2}{4}$ is an integer, then $\mathcal{T}(G)$ exhibits Pair-LPGST between $\textbf{e}^{n+m}_a - \textbf{e}^{n+m}_b$ and $\textbf{e}^{n+m}_c - \textbf{e}^{n+m}_d$.
\end{theorem}    
\begin{proof}
Let $\Phi_{ab}$ be the Laplacian eigenvalue support of $\textbf{e}^n_a - \textbf{e}^n_b$ in $G$. It is clear that $0 \notin \Phi_{ab}$. Suppose that there exists an integer $x$ and a square-free integer $\Delta$ strictly greater than 1 such that $\Phi_{ab} \subseteq \left\{\frac{1}{2}(x + y \sqrt{\Delta}) \colon ~y~\mathrm{is~an~integer}\right \}$. Since $G$ exhibits Pair-LPST between $\textbf{e}^n_a - \textbf{e}^n_b$ and $\textbf{e}^n_c - \textbf{e}^n_d$ at $\frac{\pi}{2}$, by Theorem~\ref{Conditions for the existence of perfect pair state transfer} and the comment after Theorem~\ref{More conditions for the existence of perfect pair state transfer}, we find that $\frac{\pi}{2}$ is a rational multiple of $\frac{\pi}{g \sqrt{\Delta}}$, where $g$ is as in Theorem~\ref{Conditions for the existence of perfect pair state transfer}. This implies that $\Delta=1$, which is a contradiction. Therefore, Theorem~\ref{Conditions for the existence of perfect pair state transfer} yields that  the eigenvalues in $\Phi_{ab}$ are integers. 

Note that $\Delta_j = \sqrt{(r + 2)^2 - 4 \theta_j}$ for $0 \leq j \leq p$. Therefore $\Delta^2_j$ is an integer for all $\theta_j \in \Phi_{ab}$. For each $\theta_j \in \Phi_{ab}$, write $\Delta_j = x_j \sqrt{y_j}$ for some integer $x_j$ and square free integer $y_j$. Since $r+1$ is not a Laplacian eigenvalue of $G$, as in the proof of Lemma~\ref{non-existence of PST 1st Lemma} we find that $\Delta_j$ is irrational for all $\theta_j \in \Phi_{ab}$. This implies that $y_j > 1$ for all $\theta_j \in \Phi_{ab}$. Then Lemma~\ref{Linearly independent corollary} yields that the set $\{1\}\cup \{\sqrt{y_j} \colon \theta_j \in \Phi_{ab}\}$ is linearly independent over $\mathbb{Q}$. Now applying Kronecker's approximation theorem, for each $\theta_j \in \Phi_{ab}$, there exists $\ell, s_j \in \mathbb{Z}$  such that
$$\ell \sqrt{y_j} - s_j \approx -\frac{\sqrt{y_j}}{8}~~\mathrm{for~all}~\theta_j \in  \Phi_{ab}.$$
This implies that
$$\left(4 \ell + \frac{1}{2}\right)\Delta_j \approx 4x_j s_j~~\mathrm{for~all}~\theta_j \in  \Phi_{ab}.$$
Let $t=\left(4\ell+\frac{1}{2}\right)\pi$. Since the cosine and sine functions are uniformly continuous on $\mathbb{R}$, we have 
$$\cos\left(\frac{\Delta_j t}{2}\right) \approx 1~~\mathrm{and}~~\sin\left(\frac{\Delta_j t}{2}\right) \approx 0~~\mathrm{for}~0 \leq j \leq p~\mathrm{with}~\theta_j \in  \Phi_{ab}.$$   
From part $(a)$ of Lemma~\ref{Total graph entry pair state}, we obtain 
\begin{align*}
&\frac{1}{2}{\left(\textbf{e}^{n+m}_a - \textbf{e}^{n+m}_b\right)}^TU_{\mathcal{T}(G)}(t)\left(\textbf{e}^{n+m}_c - \textbf{e}^{n+m}_d\right) \\
=& \frac{1}{2}\sum_{\theta_j \in \Phi_{ab}}\exp(-\textbf{i}t((r + 2\theta_j + 2)/2)){\left(\textbf{e}^n_{a} - \textbf{e}^n_{b}\right)}^TE_{\theta_j}(G)\left(\textbf{e}^n_{c} - \textbf{e}^n_{d}\right)\\
&\times \left(\cos\left(\frac{\Delta_j t}{2}\right) + \textbf{i}\frac{2 - r}{\Delta_j}\sin\left(\frac{\Delta_j t}{2}\right)\right).
\end{align*}
Since $\frac{r+2}{4}$ and $\theta_j \in  \Phi_{ab}$ are integers, and $t=\left(4\ell+\frac{1}{2}\right)\pi$, we find that 
$$\exp(-\textbf{i}t((r + 2\theta_j + 2)/2))= {(-1)}^{\frac{r+2}{4}}\exp(-\textbf{i}(\pi/2)\theta_j.$$
Therefore  
\begin{align*}
&\frac{1}{2}{\left(\textbf{e}^{n+m}_a - \textbf{e}^{n+m}_b\right)}^TU_{\mathcal{T}(G)}(t)\left(\textbf{e}^{n+m}_c - \textbf{e}^{n+m}_d\right) \\
\approx & {(-1)}^{\frac{r+2}{4}} \frac{1}{2}\sum_{\theta_j \in  \Phi_{ab}}\exp(-\textbf{i}(\pi/2)\theta_j){\left(\textbf{e}^n_{a} - \textbf{e}^n_{b}\right)}^TE_{\theta_j}(G)\left(\textbf{e}^n_{c} - \textbf{e}^n_{d}\right) \\
 =& {(-1)}^{\frac{r+2}{4}} \frac{1}{2} \sum_{j=0}^{p}\exp(-\textbf{i}(\pi/2)\theta_j){\left(\textbf{e}^n_{a} - \textbf{e}^n_{b}\right)}^TE_{\theta_j}(G)\left(\textbf{e}^n_{c} - \textbf{e}^n_{d}\right)\\
 =& {(-1)}^{\frac{r+2}{4}} \frac{1}{2}{\left(\textbf{e}^n_{a} - \textbf{e}^n_{b}\right)}^T \left(\sum_{j=0}^{p}\exp(-\textbf{i}(\pi/2)\theta_j)E_{\theta_j}(G)\right)\left(\textbf{e}^n_{c} - \textbf{e}^n_{d}\right)\\
=& {(-1)}^{\frac{r+2}{4}} \frac{1}{2}{\left(\textbf{e}^n_{a} - \textbf{e}^n_{b}\right)}^T U_G(\pi/2)\left(\textbf{e}^n_{c} - \textbf{e}^n_{d}\right).
\end{align*}
Since $G$ exhibits Pair-LPST between ${\textbf{e}^n_{a}} - {\textbf{e}^n_{b}}$ and ${\textbf{e}^n_{c}} - {\textbf{e}^n_{d}}$ at $\frac{\pi}{2}$, we have 
$$\left|\frac{1}{2}{({\textbf{e}^n_{a}} - {\textbf{e}^n_{b}})}^TU_G(\pi/2)({\textbf{e}^n_{c}} - {\textbf{e}^n_{d}})\right| = 1.$$
Thus it follows that  
$$\left|\frac{1}{2}{\left(\textbf{e}^{n+m}_a - \textbf{e}^{n+m}_b\right)}^TU_{\mathcal{T}(G)}(t)\left(\textbf{e}^{n+m}_c - \textbf{e}^{n+m}_d\right)\right| \approx 1.$$
Hence, $\mathcal{T}(G)$ exhibits Pair-LPGST between the pair states $\textbf{e}^{n+m}_a - \textbf{e}^{n+m}_b$ and $\textbf{e}^{n+m}_c - \textbf{e}^{n+m}_d$.
\end{proof}

Let $G$ be an $r$-regular connected bipartite graph, where $r \geq 3$ and $X_1 \cup X_2$ be a bipartition of the vertex set of $G$. Then $2r$ is a Laplacian eigenvalue of $G$ and the Laplacian eigenprojector of $G$ corresponding to $2r$ is given by
\begin{equation}\label{f_1}
E_{2r}(G) = \frac{1}{n}\begin{pmatrix}
         J_{|X_1|\times |X_1|}  &  -J_{|X_1|\times |X_2|}\\
         -J_{|X_2|\times |X_1|} &  J_{|X_2|\times |X_2|}
         \end{pmatrix}.
\end{equation}
The next theorem provides sufficient conditions for the existence of Pair-LPGST on the total graph of a regular connected bipartite graph.
\begin{theorem}\label{Second theorem PGST}
Let $G$ be an $r$-regular connected bipartite graph such that $r\geq3$. Let $X_1 \cup X_2$ be a bipartition of the vertex set of $G$. Also, let $G$ exhibit Pair-LPST between $\textbf{e}^n_a - \textbf{e}^n_b$ and $\textbf{e}^n_c - \textbf{e}^n_d$ at $\frac{\pi}{2}$ and $r+1$ not be a Laplacian eigenvalue of $G$.   
\begin{enumerate}
\item[(a)] Suppose that $a, b \in X_1$ or $a, b \in X_2$. If $\frac{r+2}{4}$ is an integer, then $\mathcal{T}(G)$ exhibits Pair-LPGST between $\textbf{e}^{n+m}_a - \textbf{e}^{n+m}_b$ and $\textbf{e}^{n+m}_c - \textbf{e}^{n+m}_d$.
\item[(b)] Suppose that $a \in X_1,~b \in X_2$ or $b \in X_1,~a \in X_2$. If $\frac{3r}{2}$ and $\frac{5r+2}{4}$ are integers of the same parity, then $\mathcal{T}(G)$ exhibits Pair-LPGST between $\textbf{e}^{n+m}_a - \textbf{e}^{n+m}_b$ and $\textbf{e}^{n+m}_c - \textbf{e}^{n+m}_d$. 
\end{enumerate}
\end{theorem}
\begin{proof}
Following the notations and the proof of Theorem~\ref{First theorem PGST}, we obtain that
$$\cos\left(\frac{\Delta_j t}{2}\right) \approx 1~~\mathrm{and}~~\sin\left(\frac{\Delta_j t}{2}\right) \approx 0~~\mathrm{for}~0 \leq j \leq p-1~\mathrm{with}~\theta_j \in \Phi_{ab},~~\mathrm{where}~t=\left(4\ell+\frac{1}{2}\right)\pi.$$   

\noindent $(a)$ If $a, b \in X_1$ or $a, b \in X_2$, then Equation~(\ref{f_1}) implies that $2r$ does not belong to $\Phi_{ab}$. Since $G$ exhibits Pair-LPST between ${\textbf{e}^n_{a}} - {\textbf{e}^n_{b}}$ and ${\textbf{e}^n_{c}} - {\textbf{e}^n_{d}}$ at $\frac{\pi}{2}$, we have 
$$\left|\frac{1}{2}{({\textbf{e}^n_{a}} - {\textbf{e}^n_{b}})}^TU_G(\pi/2)({\textbf{e}^n_{c}} - {\textbf{e}^n_{d}})\right| = 1.$$
Now from part $(b)$ of Lemma~\ref{Total graph entry pair state} and proceeding as in the proof of Theorem~\ref{First theorem PGST}, we have 
$$\left|\frac{1}{2}{\left(\textbf{e}^{n+m}_a - \textbf{e}^{n+m}_b\right)}^TU_{\mathcal{T}(G)}(t)\left(\textbf{e}^{n+m}_c - \textbf{e}^{n+m}_d\right) \right| \approx \left| \frac{1}{2}{\left(\textbf{e}^n_{a} - \textbf{e}^n_{b}\right)}^T U_G(\pi/2)\left(\textbf{e}^n_{c} - \textbf{e}^n_{d}\right) \right|=1.$$
Hence, $\mathcal{T}(G)$ exhibits Pair-LPGST between the pair states $\textbf{e}^{n+m}_a - \textbf{e}^{n+m}_b$ and $\textbf{e}^{n+m}_c - \textbf{e}^{n+m}_d$.

\noindent $(b)$ If $a \in X_1,~b \in X_2$ or $b \in X_1,~a \in X_2$, then Equation~(\ref{f_1}) implies that $2r$ belongs to $\Phi_{ab}$. Now 
\begin{align*}
&\frac{1}{2}{\left(\textbf{e}^{n+m}_a - \textbf{e}^{n+m}_b\right)}^TU_{\mathcal{T}(G)}(t)\left(\textbf{e}^{n+m}_c - \textbf{e}^{n+m}_d\right)\\
=&\frac{1}{2} \sum_{\theta_j \in \Phi_{ab} \setminus \{2r\}}
\exp(-\textbf{i}t((r+2\theta_j+2)/2))
{\left(\textbf{e}^n_{a} - \textbf{e}^n_{b}\right)}^TE_{\theta_j}(G)\left(\textbf{e}^n_{c} - \textbf{e}^n_{d}\right)\\
&~\times\left(\cos\left(\frac{\Delta_j t}{2}\right) + \textbf{i}\frac{2-r}{\Delta_j}\sin\left(\frac{\Delta_j t}{2}\right)\right) + \frac{1}{2} \exp(-3\textbf{i}tr){\left(\textbf{e}^n_{a} - \textbf{e}^n_{b}\right)}^TE_{2r}(G)\left(\textbf{e}^n_{c} - \textbf{e}^n_{d}\right).
\end{align*}
Since $\frac{3r}{2}$ and $\frac{5r+2}{4}$ are integers of the same parity, $\theta_j \in  \Phi_{ab}$ are integers, $t=\left(4\ell+\frac{1}{2}\right)\pi$ and $\theta_p=2r$, we find that 
$$\exp(-3\textbf{i}tr)=\exp(-\textbf{i}t((r+2\theta_p+2)/2)).$$
Further, $\frac{r+2}{4}$ is an odd integer. Therefore 
\begin{align*}
&\frac{1}{2}{\left(\textbf{e}^{n+m}_a - \textbf{e}^{n+m}_b\right)}^TU_{\mathcal{T}(G)}(t)\left(\textbf{e}^{n+m}_c - \textbf{e}^{n+m}_d\right)\\
\approx & \frac{1}{2} \sum_{\theta_j \in \Phi_{ab}}
\exp(-\textbf{i}t((r+2\theta_j+2)/2))
{\left(\textbf{e}^n_{a} - \textbf{e}^n_{b}\right)}^TE_{\theta_j}(G)\left(\textbf{e}^n_{c} - \textbf{e}^n_{d}\right)\\
=&-\frac{1}{2}{\left(\textbf{e}^n_{a} - \textbf{e}^n_{b}\right)}^TU_G(\pi /2)     {\left(\textbf{e}^n_{c} - \textbf{e}^n_{d}\right)}.
\end{align*}
Since $G$ exhibits Pair-LPST between ${\textbf{e}^n_{a}} - {\textbf{e}^n_{b}}$ and ${\textbf{e}^n_{c}} - {\textbf{e}^n_{d}}$ at $\frac{\pi}{2}$, we have 
$$\left|\frac{1}{2}{({\textbf{e}^n_{a}} - {\textbf{e}^n_{b}})}^TU_G(\pi/2)({\textbf{e}^n_{c}} - {\textbf{e}^n_{d}})\right| = 1.$$
Thus 
$$\left|\frac{1}{2}{\left(\textbf{e}^{n+m}_a - \textbf{e}^{n+m}_b\right)}^TU_{\mathcal{T}(G)}(t)\left(\textbf{e}^{n+m}_c - \textbf{e}^{n+m}_d\right)\right| \approx 1.$$
Hence, $\mathcal{T}(G)$ exhibits Pair-LPGST between the pair states $\textbf{e}^{n+m}_a - \textbf{e}^{n+m}_b$ and $\textbf{e}^{n+m}_c - \textbf{e}^{n+m}_d$.
\end{proof}

\section{Examples}\label{Concrete examples}
In this section, we obtain several infinite families of total graphs exhibiting Pair-LPGST using Theorem~\ref{First theorem PGST} and Theorem~\ref{Second theorem PGST}. We also find that the total graph of the complete graph $K_n$ exhibits Pair-LPGST if and only if $n=3$.
\begin{ex}\label{Example 1}
Let $m$ be an even positive integer. Then the total graph of the cocktail party graph $\overline{\cup_m K_2}$ exhibits Pair-LPGST. 
\end{ex}
\begin{proof}
Observe that $\overline{\cup_m K_2}=\mathrm{Cay}(\mathbb{Z}_{2m}, \mathbb{Z}_{2m} \setminus \{0,m\})$. The graph $\overline{\cup_m K_2}$ is an $r$-regular non-bipartite graph, where $r=2m-2$, and the distinct Laplacian eigenvalues are $0,2m-2,2m$. From Theorem~\ref{Distance regular graphs}, it follows that $\overline{\cup_m K_2}$ exhibits LPST at $\frac{\pi}{2}$. Therefore by Theorem~\ref{PST implies Pair PST}, $\overline{\cup_m K_2}$ exhibits Pair-LPST at $\frac{\pi}{2}$. Also, $r+1$ is not a Laplacian eigenvalue of $\overline{\cup_m K_2}$ and $\frac{r+2}{4}$ is an integer. Now Theorem~\ref{First theorem PGST} implies that the total graph of $\overline{\cup_m K_2}$ exhibits Pair-LPGST. 
\end{proof}

\begin{ex}\label{Example 2}
Let $n=4k$, where $k$ is an odd integer. Then the total graph of the Cayley graph $\mathrm{Cay}(D_{2n}, S)$ exhibits Pair-LPGST, where $S = \{u^k,u^{3k}\} \cup v \langle u \rangle$.
\end{ex}
\begin{proof}
Clearly, $\mathrm{Cay}(D_{2n}, S)$ is an $r$-regular non-bipartite graph, where $r=4k+2$. From Cao and Feng~\cite{CaoDihedral2021}, we find that 
the distinct Laplacian eigenvalues of $\mathrm{Cay}(D_{2n}, S)$ are $0,4k,4k+2,4k+4,8k$, and that it exhibits LPST at $\frac{\pi}{2}$. Therefore by Theorem~\ref{PST implies Pair PST}, $\mathrm{Cay}(D_{2n}, S)$ exhibits Pair-LPST at $\frac{\pi}{2}$. Also, $r+1$ is not a Laplacian eigenvalue of $\mathrm{Cay}(D_{2n}, S)$ and $\frac{r+2}{4}$ is an integer. Now Theorem~\ref{First theorem PGST} implies that the total graph of $\mathrm{Cay}(D_{2n}, S)$ exhibits Pair-LPGST. 
\end{proof}

\begin{ex}\label{Example 3}
Let $n=2p$, where $p$ is an odd prime. Then the total graph of $\mathrm{Cay}(D_{2n}, S)$ exhibits Pair-LPGST, where $S =  \left(\langle u \rangle \setminus \{\textbf{1},u^p\}\right)\cup v \langle u \rangle$.
\end{ex}
\begin{proof}
Clearly, $\mathrm{Cay}(D_{2n}, S)$ is an $r$-regular non-bipartite graph, where $r=4p-2$. From Cao and Feng~\cite{CaoDihedral2021}, we find that the distinct Laplacian eigenvalues of $\mathrm{Cay}(D_{2n}, S)$ are $0,4p-2,4p$, and that it exhibits LPST at $\frac{\pi}{2}$. Therefore by Theorem~\ref{PST implies Pair PST}, $\mathrm{Cay}(D_{2n}, S)$ exhibits Pair-LPST at $\frac{\pi}{2}$. Also, $r+1$ is not a Laplacian eigenvalue of $\mathrm{Cay}(D_{2n}, S)$ and $\frac{r+2}{4}$ is an integer. Therefore, Theorem~\ref{First theorem PGST} implies that the total graph of $\mathrm{Cay}(D_{2n}, S)$ exhibits Pair-LPGST. 
\end{proof}

\begin{ex}\label{Example 4}
Let $n$ be an integer such that $n \geq 2$. Then the total graph of $\mathrm{Cay}(T_{4n}, S)$ exhibits Pair-LPGST, where $S = T_{4n} \setminus \{\textbf{1},u^n\}$. 
\end{ex}
\begin{proof}
Clearly, $\mathrm{Cay}(T_{4n}, S)$ is an $r$-regular non-bipartite graph, where $r=4n-2$. From Arezoomand et al.~\cite{ArezoomandDicyclic2022}, we find that the distinct Laplacian eigenvalues of $\mathrm{Cay}(T_{4n}, S)$ are $0,4n-2,4n$, and that it exhibits LPST at $\frac{\pi}{2}$.  Therefore by Theorem~\ref{PST implies Pair PST}, $\mathrm{Cay}(T_{4n}, S)$ exhibits Pair-LPST at $\frac{\pi}{2}$.
Also, $r+1$ is not a Laplacian eigenvalue of $\mathrm{Cay}(T_{4n}, S)$ and $\frac{r+2}{4}$ is an integer. Therefore, Theorem~\ref{First theorem PGST} implies that the total graph of $\mathrm{Cay}(T_{4n}, S)$ exhibits Pair-LPGST. 
\end{proof}

\begin{ex}\label{Example 5}
Let $n$ be an odd integer such that $n \geq 3$. Then the total graph of $\mathrm{Cay}(SD_{8n}, S)$ exhibits Pair-LPGST, where $S = \{u^n,u^{3n}\} \cup v \langle u \rangle$. 
\end{ex}
\begin{proof}
Clearly, $\mathrm{Cay}(SD_{8n}, S)$ is an $r$-regular non-bipartite graph, where $r=4n+2$. From Luo et al.~\cite{LuoSemidihedral2022}, we find that the distinct Laplacian eigenvalues of $\mathrm{Cay}(SD_{8n}, S)$ are $0,4n,4n+2,4n+4,8n$, and that it exhibits LPST at $\frac{\pi}{2}$.  Therefore by Theorem~\ref{PST implies Pair PST}, $\mathrm{Cay}(SD_{8n}, S)$ exhibits Pair-LPST at $\frac{\pi}{2}$. Also, $r+1$ is not a Laplacian eigenvalue of $\mathrm{Cay}(SD_{8n}, S)$ and $\frac{r+2}{4}$ is an integer. Therefore, Theorem~\ref{First theorem PGST} implies that the total graph of $\mathrm{Cay}(SD_{8n}, S)$ exhibits Pair-LPGST. 
\end{proof}

\begin{ex}\label{Example 6}
Let $r=8m_1+2$, where $m_1$ is a positive integer. Then the total graph of the $r$-cube $Q_r$ exhibits Pair-LPGST. 
\end{ex}
\begin{proof}
Clearly, $Q_r$ is an $r$-regular bipartite graph. Since $r$ is even, $r+1$ is not a Laplacian eigenvalue of $Q_r$. It follows from Christandl et al.~\cite{Christandl2004} and Theorem~\ref{PST implies Pair PST} that $Q_r$ exhibits Pair-LPST at $\frac{\pi}{2}$. Further, $\frac{r+2}{4}=2m_1+1, \frac{3r}{2}=12m_1+3$ and $\frac{5r+2}{4}=10m_1+3$. Therefore by Theorem~\ref{Second theorem PGST}, the total graph of $Q_r$ exhibits Pair-LPGST.
\end{proof}

\begin{ex}\label{Example 7}
Let $m_1$ and $m_2$ be two positive integers. Also, let $G_1=\overline{\cup_{4m_1+2} K_2}$ and $G_2=Q_{8m_2}$. Then the total graph of $G_1 \square G_2$ exhibits Pair-LPGST.
\end{ex}
\begin{proof}
Clearly, $G_1 \square G_2$ is an $r$-regular non-bipartite graph, where $r=8(m_1+m_2)+2$. Since $G_1$ and $G_2$ are Cayley graphs, Theorem~\ref{Distance regular graphs}, Christandl et al.~\cite{Christandl2004}, Theorem~\ref{Cayley Graph Cartesian Product} and Theorem~\ref{PST and Cartesian product} altogether imply that $G_1 \square G_2$ exhibits Pair-LPST at $\frac{\pi}{2}$. Also, $r+1$ is not a Laplacian eigenvalue of $G_1 \square G_2$ and $\frac{r+2}{4}$ is an integer. Therefore by Theorem~\ref{First theorem PGST}, the total graph of $G_1 \square G_2$ exhibits Pair-LPGST.
\end{proof}
In Table 1, we list a few more total graphs exhibiting Pair-LPGST using Cartesian product.

\begin{ex}\label{Example 8}
Let $m_1$ be a positive integer and $m_2$ be an even integer such that $m_2 \geq 4$. Also, let $G_1=\overline{\cup_{4m_1+2} K_2}$ and $G_2=K_{m_2}$. Then the total graph of $G_1 \times G_2$ exhibits Pair-LPGST.
\end{ex}
\begin{proof}
Clearly, $G_1 \times G_2$ is an $r$-regular non-bipartite graph, where $r= 8m_1(m_2-1)+2m_2-2$. Since $G_1$ exhibits PST at $\frac{\pi}{2}$, the eigenvalues of $G_1$ are even and $G_2$ is a circulant graph with odd eigenvalues, Theorem~\ref{Cayley Graph Tensor Product} and Theorem~\ref{PST and tensor product} together imply that $G_1 \times G_2$ exhibits Pair-LPST at $\frac{\pi}{2}$.  Also, $r+1$ is not a Laplacian eigenvalue of $G_1 \times G_2$ and $\frac{r+2}{4}$ is an integer. Therefore by Theorem~\ref{First theorem PGST}, the total graph of $G_1 \times G_2$ exhibits Pair-LPGST.
\end{proof}

\begin{ex}\label{Example 9}
Let $m_1$ be an odd positive integer and $m_2$ be a positive integer such that $m_2 \equiv 0~(\mathrm{mod}~4)$. Also, let $G_1=Q_{4m_1+2}$ and $G_2= K_{m_2}$. Then the total graph of $G_1 \times G_2$ exhibits Pair-LPGST. 
\end{ex}
\begin{proof}
Clearly, $G_1 \times G_2$ is an $r$-regular bipartite graph, where $r=4m_1(m_2-1)+2m_2-2$ and $r+1$ is not a Laplacian eigenvalue of $G_1 \times G_2$. Also, $G_1 \times G_2$ exhibits Pair-LPST at $\frac{\pi}{2}$. Since $\frac{r+2}{4},~\frac{3r}{2}$ and $\frac{5r+2}{4}$ are odd integers, Theorem~\ref{Second theorem PGST} yields that the total graph of $G_1 \times G_2$ exhibits Pair-LPGST.
\end{proof}
In Table 2, we list a few more total graphs exhibiting Pair-LPGST using tensor product. It is worth mentioning that the total graphs appeared in Example~\ref{Example 1} to Example~\ref{Example 9}, Table 1 and Table 2 provide several infinite families of total graphs of Laplacian integral graphs exhibiting Pair-LPGST that fail to exhibit Pair-LPST.  

\begin{table}\label{First Table}
\caption{Total graphs of Cartesian product of graphs exhibiting Pair-LPGST}
\centering
\begin{tabular}{|c| c| c|}
\hline
$G_1$ & $G_2$ & Total graph of $G_1 \square G_2$\\
\hline
\makecell{The Cayley graph in Example~\ref{Example 2}} & \makecell{$Q_{4m_1+4}$, where $m_1 \in \mathbb{N} \cup \{0\}$} & Exhibits Pair-LPGST\\
\hhline{|-~-|}
\makecell{The Cayley graph in Example~\ref{Example 3}} &  & Exhibits Pair-LPGST\\
\hhline{|-~-|}
\makecell{The Cayley graph in Example~\ref{Example 4}} &  & Exhibits Pair-LPGST\\
\hhline{|-~-|}
\makecell{The Cayley graph in Example~\ref{Example 5}} &  & Exhibits Pair-LPGST\\
\hline
\end{tabular}
\end{table}

\begin{table}\label{Second Table}
\caption{Total graphs of tensor product of graphs exhibiting Pair-LPGST}
\centering
\begin{tabular}{|c| c| c|}
\hline
$G_1$ & $G_2$ & Total graph of $G_1 \times G_2$\\
\hline
\makecell{The Cayley graph in Example~\ref{Example 2}} & \makecell{$K_{m_2}$,  where $m_2$ is even and $m_2 \geq 4$} & Exhibits Pair-LPGST\\
\hhline{|-~-|}
\makecell{The Cayley graph in Example~\ref{Example 3}} &  & Exhibits Pair-LPGST\\
\hhline{|-~-|}
\makecell{The Cayley graph in Example~\ref{Example 4}} &  & Exhibits Pair-LPGST\\
\hhline{|-~-|}
\makecell{The Cayley graph in Example~\ref{Example 5}} &  & Exhibits Pair-LPGST\\
\hline
\end{tabular}
\end{table}

Now we consider Pair-LPGST on total graph of the complete graph $K_n$. Since $\mathcal{T}(K_2)$ is the cycle on three vertices, Lemma 3 of Pal and Mohapatra~\cite{PalMohapatra2026} implies that the graph $\mathcal{T}(K_2)$ does not exhibit Pair-LPGST. Note that $\mathcal{T}(K_3)= \overline{\cup_3 K_2} = \mathrm{Cay}(\mathbb{Z}_{6}, \{1, 2, 4, 5\})$. It can be verified that the Laplacian eigenvalues of $\overline{\cup_3 K_2}$ are given by $\theta_0=0, \theta_1=4$ and $\theta_2=6$. Further, the Laplacian eigenprojectors of $\overline{\cup_3 K_2}$ are $\frac{1}{6} J_{6 \times 6}, E_{{\theta_1}}(\overline{\cup_3 K_2})$ and $E_{{\theta_2}}(\overline{\cup_3 K_2})$, where 
\begin{align*}
\resizebox{0.999\hsize}{!}{$
E_{{\theta_1}}(\overline{\cup_3 K_2}) = \frac{1}{2}\begin{pmatrix}
         1 & -1 & 0 & 0 & 0 & 0  \\
         -1 & 1 & 0 & 0 & 0 & 0  \\
         0 & 0 & 1 & -1 & 0 & 0  \\
         0 & 0 & -1 & 1 & 0 & 0  \\
         0 & 0 & 0 & 0 & 1 & -1 \\
         0 & 0 & 0 & 0 & -1 & 1  
         \end{pmatrix} ~\mathrm{and} ~E_{{\theta_2}}(\overline{\cup_3 K_2}) = \frac{1}{6}\begin{pmatrix}
         2 & 2 & -1 & -1 & -1 & -1 \\
         2 & 2 & -1 & -1 & -1 & -1 \\
         -1 & -1 & 2 & 2 & -1 & -1 \\
         -1 & -1 & 2 & 2 & -1 & -1 \\
         -1 & -1 & -1 & -1 & 2 & 2 \\
         -1 & -1 & -1 & -1 & 2 & 2 
         \end{pmatrix}.$}   
\end{align*}

Consider the vertices $a,b, c$ and $d$ of $\overline{\cup_3 K_2}$ such that $a=0, b=4, c=1$ and $d=5$. Then we obtain that the pair states $\textbf{e}^{6}_a - \textbf{e}^{6}_b$ and $\textbf{e}^{6}_c - \textbf{e}^{6}_d$ of the graph $\overline{\cup_3 K_2}$ are Laplacian strongly cospectral such that $\Phi^{+}_{ab,cd} = \{\theta_2\}$ and $\Phi^{-}_{ab,cd}=\{\theta_1\}$. Let $\ell, m \in \mathbb{Z}$ such that $\ell \theta_2 + m \theta_1 =0$ and  
 $\ell + m =0$. Then $m=0$. Therefore, Theorem~\ref{Paths} yields that $\overline{\cup_3 K_2}$ exhibits Pair-LPGST between $\textbf{e}^{6}_a - \textbf{e}^{6}_b$ and $\textbf{e}^{6}_c - \textbf{e}^{6}_d$.

Now we consider the case that $n \geq 4$. Note that $0$ and $n$ are the only Laplacian eigenvalues of $K_n$, that is, $\theta_0 = 0$ and $\theta_1 = n$. The Laplacian eigenprojectors of $K_n$ corresponding to $\theta_0$ and $\theta_1$ are given by 
$$E_{\theta_0}(K_n) = \frac{1}{n}J_{n \times n}~~\mathrm{and}~~E_{\theta_1}(K_n) = I_{n \times n} - \frac{1}{n}J_{n \times n},$$
respectively.
Let $\{z_1,\ldots,z_{m-n}\}$ be an orthonormal basis of the null-space of $R_{K_n}$, where $m=\frac{n^2-n}{2}$. Applying Theorem~\ref{Total graph spectra}, we obtain the Laplacian eigenvalues of $\mathcal{T}(K_n)$ as follows:
$$\theta_0^{+}=n+1,~\theta_0^{-}=0,~\theta_1^{+}=2n,~\theta_1^{-}=n+1 ~~\mathrm{and}~~2n.$$
Further, the Laplacian eigenprojectors of $\mathcal{T}(K_n)$ corresponding to $\theta_0^{+},\theta_0^{-},\theta_1^{+},\theta_1^{-}$ and $2n$ are 
\begin{align}
E_{\theta_0^{+}}(\mathcal{T}(K_n)) &= \frac{1}{n^2-1}\begin{pmatrix}
         \frac{n^2-2n+1}{n} J_{n\times n} &  \frac{-2n+2}{n} J_{n \times m}\\
          \frac{-2n+2}{n}J_{m \times n} &  \frac{4}{n} J_{m \times m}
       \end{pmatrix},\label{v1}\\
E_{\theta_0^{-}}{(\mathcal{T}(K_n))} &= \frac{2}{n^2+n}J_{(n+m) \times (n+m)},\label{v2}\\
E_{\theta_1^{+}}(\mathcal{T}(K_n)) &= \frac{1}{n^2 - 3n + 2}\begin{pmatrix}
         {(2-n)}^2{\left(I_{n \times n} - \frac{1}{n}J_{n \times n}\right)}  &  (2-n){\left(R_{K_n}-\frac{2}{n} J_{n \times m}\right)}\\
            (2-n){\left(R^T_{K_n}-\frac{2}{n} J_{m \times n}\right)} & R^T_{K_n} R_{K_n} - \frac{4}{n}J_{m \times m} 
       \end{pmatrix},\label{v3}\\
E_{\theta_1^{-}}(\mathcal{T}(K_n)) &= \frac{1}{n-1}\begin{pmatrix}
         I_{n \times n} - \frac{1}{n}J_{n \times n}  &  R_{K_n}-\frac{2}{n} J_{n \times m}\\
           R^T_{K_n}-\frac{2}{n} J_{m \times n} & R^T_{K_n} R_{K_n} - \frac{4}{n}J_{m \times m} 
       \end{pmatrix}~\mathrm{and}\label{v4}\\
E_{2n}{(\mathcal{T}(K_n))} &= \begin{pmatrix}
         \textbf{0}_{n \times n} & \textbf{0}_{n \times m}  \\
         \textbf{0}_{m \times n} & \sum_{i=1}^{m-n} z_iz^T_i  
       \end{pmatrix},\label{v5} 
\end{align}
respectively. In the following theorem, we prove that if a complete graph contains more than three vertices, then the corresponding total graph does not exhibit Pair-LPGST.
\begin{ex}\label{Total graphs of complete graphs}
Let $n$ be an integer with $n \geq 4$. Then the total graph of $K_n$ does not exhibit Pair-LPGST.
\end{ex}
\begin{proof}
Let $\mathcal{T}(K_n)$ exhibit Pair-LPGST between $\textbf{e}^{n+m}_a - \textbf{e}^{n+m}_b$ and $\textbf{e}^{n+m}_c - \textbf{e}^{n+m}_d$. Since $n \geq 4$, the graph $K_n$ is non-bipartite. Consider the following three cases.

\noindent \textbf{Case 1.} Let $a$ and $b$ be vertices of $K_n$. Then from equations (\ref{v1}), (\ref{v2}) and (\ref{v5}), we obtain 
$$E_{\theta_0^{+}}{(\mathcal{T}(K_n))}(\textbf{e}^{n+m}_a - \textbf{e}^{n+m}_b) = E_{\theta_0^{-}}{(\mathcal{T}(K_n))}(\textbf{e}^{n+m}_a - \textbf{e}^{n+m}_b) = E_{2n}{(\mathcal{T}(K_n))}(\textbf{e}^{n+m}_a - \textbf{e}^{n+m}_b)= \textbf{0}_{n+m}.$$
Since the pair state $\textbf{e}^{n+m}_a - \textbf{e}^{n+m}_b$ involves in Pair-LPGST, Theorem~\ref{fixed state} yields that $\Phi_{ab} = \{{\theta_1}^{+},{\theta_1}^{-}\}$. Then by Lemma~\ref{perfect state transfer lemma non-bipartite}, either ${\theta_1}^{+},{\theta_1}^{-} \in \Phi^{+}_{ab,cd}$ or ${\theta_1}^{+},{\theta_1}^{-} \in \Phi^{-}_{ab,cd}$, that is, either $\Phi^{-}_{ab,cd}$ is empty or $\Phi^{+}_{ab,cd}$ is empty, which contradicts Lemma~\ref{fixed state}.

\noindent \textbf{Case 2.} Let $a$ and $b$ be edges of $K_n$. Then from equations (\ref{v1}) and (\ref{v2}), we obtain 
$$E_{\theta_0^{+}}{(\mathcal{T}(K_n))}(\textbf{e}^{n+m}_a - \textbf{e}^{n+m}_b) = E_{\theta_0^{-}}{(\mathcal{T}(K_n))}(\textbf{e}^{n+m}_a - \textbf{e}^{n+m}_b) = \textbf{0}_{n+m}.$$
Therefore $\Phi_{ab} \subseteq \{{\theta_1}^{+},{\theta_1}^{-}, 2n\}$. By Lemma~\ref{perfect state transfer lemma non-bipartite}, we have either ${\theta_1}^{+},{\theta_1}^{-} \in \Phi_{ab}$ or ${\theta_1}^{+},{\theta_1}^{-} \notin \Phi_{ab}$. Also, by Lemma~\ref{fixed state}, $|\Phi_{ab}| \geq 2$. Therefore, $\Phi_{ab} = \{{\theta_1}^{+},{\theta_1}^{-}\}$ or $\Phi_{ab} = \{{\theta_1}^{+},{\theta_1}^{-},2n\}$. If $\Phi_{ab} = \{{\theta_1}^{+},{\theta_1}^{-}\}$, then we get a contradiction as in Case 1. If $\Phi_{ab} = \{{\theta_1}^{+},{\theta_1}^{-},2n\}$, then consider the following two sub-cases.

\noindent \textbf{Sub-case 1.} Let ${\theta_1}^{+},{\theta_1}^{-} \in \Phi^{+}_{ab,cd}$ and $2n \in \Phi^{-}_{ab,cd}$. Consider the integers $\ell_1,\ell_2$ and $m_1$ given by $\ell_1=1,\ell_2=0$ and $m_1=-1$. Then $\ell_1 {\theta_1}^{+} + \ell_2 {\theta_1}^{-} + m_1\cdot 2n = 0$ and $\ell_1 + \ell_2 + m_1=0$. But $m_1$ is odd, which contradicts Theorem~\ref{Paths}. 

\noindent \textbf{Sub-case 2.} Let ${\theta_1}^{+},{\theta_1}^{-} \in \Phi^{-}_{ab,cd}$ and $2n \in \Phi^{+}_{ab,cd}$. Consider the integers $\ell_1,m_1$ and $m_2$ given by $\ell_1=1,m_1=-1$ and $m_2=0$. Then $\ell_1 \cdot 2n + m_1 {\theta_1}^{+} + m_2 {\theta_1}^{-} = 0$ and $\ell_1 + m_1 + m_2=0$. But $m_1+m_2$ is odd, which contradicts Theorem~\ref{Paths}. 

\noindent \textbf{Case 3.} Let $a$ be a vertex and $b$ be an edge of $K_n$. Since $\mathcal{T}(K_n)$ exhibits Pair-LPGST between $\textbf{e}^{n+m}_a - \textbf{e}^{n+m}_b$ and $\textbf{e}^{n+m}_c - \textbf{e}^{n+m}_d$, Theorem~\ref{Conditions for the existence of perfect pair state transfer} implies that $\textbf{e}^{n+m}_a - \textbf{e}^{n+m}_b$ and $\textbf{e}^{n+m}_c - \textbf{e}^{n+m}_d$ are Laplacian strongly cospectral. Therefore
\begin{align}
&E_{{\theta_0}^{+}}(\mathcal{T}(K_n))(\textbf{e}^{n+m}_a - \textbf{e}^{n+m}_b) = \pm E_{{\theta_0}^{+}}(\mathcal{T}(K_n))(\textbf{e}^{n+m}_c - \textbf{e}^{n+m}_d)~\mathrm{and}~\label{t_1}\\
&E_{{\theta_1}^{+}}(\mathcal{T}(K_n))(\textbf{e}^{n+m}_a - \textbf{e}^{n+m}_b) = \pm E_{{\theta_1}^{+}}(\mathcal{T}(K_n))(\textbf{e}^{n+m}_c - \textbf{e}^{n+m}_d)\label{t_2}.
\end{align}
From equation (\ref{v1}) and (\ref{t_1}), it follows that $c$ is a vertex and $d$ is an edge of $K_n$; or $d$ is a vertex and $c$ is an edge of $K_n$. It is enough to consider that $c$ is a vertex and $d$ is an edge of $K_n$. Now from Equation~(\ref{v3}), we have
\begin{align*}
&E_{{\theta_1}^{+}}(\mathcal{T}(K_n))(\textbf{e}^{n+m}_a - \textbf{e}^{n+m}_b) = \frac{1}{n^2-3n+2}\begin{pmatrix}
         (2-n)^2\textbf{e}^n_a-(2-n)R_{K_n}\textbf{e}^m_b + (2-n)\textbf{1}_n \\
         (2-n)R^T_{K_n}\textbf{e}^n_a - R^T_{K_n} R_{K_n}\textbf{e}^m_b +\textbf{1}_m  
       \end{pmatrix}~\mathrm{and}~\\
&E_{{\theta_1}^{+}}(\mathcal{T}(K_n))(\textbf{e}^{n+m}_c - \textbf{e}^{n+m}_d) = \frac{1}{n^2-3n+2}\begin{pmatrix}
         (2-n)^2\textbf{e}^n_c-(2-n)R_{K_n}\textbf{e}^m_d + (2-n)\textbf{1}_n \\
         (2-n)R^T_{K_n}\textbf{e}^n_c - R^T_{K_n} R_{K_n}\textbf{e}^m_d + \textbf{1}_m  
       \end{pmatrix}.
\end{align*}
If $E_{{\theta_1}^{+}}(\mathcal{T}(K_n))(\textbf{e}^{n+m}_a - \textbf{e}^{n+m}_b) = E_{{\theta_1}^{+}}(\mathcal{T}(K_n))(\textbf{e}^{n+m}_c - \textbf{e}^{n+m}_d)$, then $\textbf{e}^n_a - \textbf{e}^n_c = \frac{1}{2-n}R_{K_n}(\textbf{e}^m_b - \textbf{e}^m_d)$, which not possible. If  
$E_{{\theta_1}^{+}}(\mathcal{T}(K_n))(\textbf{e}^{n+m}_a - \textbf{e}^{n+m}_b) = -E_{{\theta_1}^{+}}(\mathcal{T}(K_n))(\textbf{e}^{n+m}_c - \textbf{e}^{n+m}_d)$, then we obtain that 
$R_{K_n}(\textbf{e}^m_b + \textbf{e}^m_d) = (2-n)(\textbf{e}^n_a + \textbf{e}^n_c)+2\cdot\textbf{1}_n$, which is again not possible. Therefore, the graph $\mathcal{T}(K_n)$ does not exhibit Pair-LPGST. This completes the proof.
\end{proof} 

\subsection*{Acknowledgements} The first author expresses gratitude for the support received through the Prime Minister’s Research Fellowship (PMRF), PMRF-ID: 1903283, awarded by the Government of India.


\end{document}